\documentclass[11pt]{article}
\usepackage{amsmath,amssymb,amsthm}
  \def\ZR{{\mathbb R}}

\def\beq{\begin{equation}}
\def\eeq{\end{equation}}
\def\be{\begin{equation}}
\def\ee{\end{equation}}
\def\beqar{\begin{eqnarray}}
\def\eeqar{\end{eqnarray}}
\def\ber{\begin{eqnarray}}
\def\eer{\end{eqnarray}}
\def\berb{\begin{eqnarray*}}
\def\eerb{\end{eqnarray*}}

\def\det{\mathop{\rm det}\nolimits}
\def\div{\mathop{\rm div}\nolimits}

\def\tr{\mathop{\rm tr}\nolimits}

\def\parder#1#2{\frac{\partial #1}{\partial #2}}

\def\norm#1.#2.{\|#1\|_{#2}}
\def\Norm#1.#2.{\big\|#1\big\|_{#2}}
\def\NOrm#1.#2.{\bigg\|#1\bigg\|_{#2}}
\def\NORm#1.#2.{\Big\|#1\Big\|_{#2}}
\def\NORM#1.#2.{\Bigg\|#1\Bigg\|_{#2}}

\newcommand{\eps}{\varepsilon}

\def\vec#1{{\mathchoice{\mbox{\boldmath$\displaystyle#1$}}
{\mbox{\boldmath$\textstyle#1$}}
{\mbox{\boldmath$\scriptstyle#1$}}
{\mbox{\boldmath$\scriptscriptstyle#1$}}}}

\newcommand{\sym}{\mathop{\rm sym}\nolimits}

\def \0b{{\hbox{\boldmath $0$}}}

\newcommand{\mbb}{{\hbox{\bf m}}} \newcommand{\nbb}{{\hbox{\bf n}}}
 
 \newcommand{\rbb}{{\hbox{\bf r}}}
 
\newcommand{\ubb}{{\hbox{\bf u}}} 
\newcommand{\wbb}{{\hbox{\bf w}}}

\newcommand{\ab}{\vec{a}} 

\newcommand{\eb}{\vec{e}} 
\newcommand{\gb}{\vec{g}}

 \newcommand{\rb}{\vec{r}}
 
\newcommand{\ub}{\vec{u}} \newcommand{\vb}{\vec{v}}

\newcommand{\Abb}{{\bf A}} 
 
\newcommand{\Ebb}{{\bf E}} 
\newcommand{\Gbb}{{\bf G}} 
\newcommand{\Ibb}{{\bf I}}

\newcommand{\Qbb}{{\bf Q}}

 \newcommand{\Xb}{\vec{X}}

\def \gamb{\vec{\gamma}}

\def \nub{\vec{\nu}} 
 \def \rhob{\vec{\rho}}
\def \varrhob{\vec{\varrho}}

\def \vphib{\vec{\varphi}} \def \chib{\vec{\chi}}
 
\def \Gamb{\vec{\Gamma}} 
\def \Thetab{\vec{\Theta}} 
\def \Xib{\vec{\Xi}}

\newcommand{\cC}{{\cal C}}

 \newcommand{\cP}{{\cal P}}

 \newcommand{\cV}{{\cal V}}

 \newcommand{\pt}{{\tilde p}}

 \newcommand{\Pt}{{\tilde P}}

 \newcommand{\Vt}{{\tilde V}}

\def \utb{\vec{\tilde{u}}}
\def \vtb{\vec{\tilde{v}}}

\def \Qtbb{{\bf \tilde{Q}}}

\newcounter{primjer}[section]
\newcounter{tvrdnja}[section]
\newcounter{zadatak}[section]
\usepackage{color}
\usepackage{epsfig}
\usepackage[latin1]{inputenc}
\newcommand{\eff}{{\rm eff}}

\newcommand{\lambdat}{\tilde{\lambda}}
\newcommand{\mut}{}
\newcommand{\hnaj}{1} 
\newcommand{\hnatri}{1} 
\newcommand{\koefdva}{\beta}  

\newcommand{\tcC}{{\cal \tilde{C}}}

\newcommand{\initial}{{\rm in}}
\newtheorem{theorem}{Theorem}
\newtheorem{lemma}[theorem]{Lemma}
\newtheorem{corollary}[theorem]{Corollary}

\theoremstyle{plain}
\newtheorem{remark}[theorem]{Remark}
\newtheorem{assumptions}[theorem]{Assumption}

\long\def\Hidden#1{\relax}

\let\Qtbb=\Qbb

\def\ep{\varepsilon}
\def\p{\partial}

\def\O{\Omega}

\def\o{\omega}

\newtheorem{proposition}[theorem]{Proposition}
\newcommand{\ogamb}{{\bf \overline{\gamb}}}
\newcommand{\ogamma}{\overline{\gamma}}
\newcommand{\Vred}{{V}}
\newcommand{\tVred}{{\Vt}}

\let\ub=\ubb
\let\vb=\vbb
\let\utb=\utbb
\let\vtb=\vtbb
\let\eb=\ebb

\title{Derivation of a poroelastic flexural shell model\thanks{ A.M. was partially supported by the  Programme Inter Carnot Fraunhofer from BMBF (Grant 01SF0804) and ANR.}}

\author{
{\bf Andro Mikeli\'c}
\\ Universit\'e de Lyon, CNRS UMR 5208,\\
  Universit\'e
Lyon 1, Institut Camille Jordan, \\   43, blvd. du 11 novembre 1918,
 69622 Villeurbanne Cedex, France \\ E-mail: {\tt Andro.Mikelic@univ-lyon1.fr} \and {\bf Josip Tamba\v ca}  \\ 	
Department of Mathematics\\
University of Zagreb\\ Bijeni\v cka 30,
10000 Zagreb, Croatia}

\date{\today}

\begin{document}

\maketitle

\tableofcontents

\newpage

\begin{abstract} In this paper we investigate the limit behavior of the solution to quasi-static Biot's equations in thin poroelastic flexural shells as the thickness  of the shell tends to zero and  extend  the results obtained for the poroelastic plate  by Marciniak-Czochra and Mikeli\'c in \cite{AMCAM:2012}. We choose Terzaghi's time corresponding to the shell thickness and obtain the strong convergence of the three-dimensional solid displacement, fluid pressure and total poroelastic stress to the solution of the new class of  shell equations.

The  derived  bending equation is coupled with the pressure equation and it contains the bending moment due to the variation in pore pressure across the shell thickness. The  effective  pressure equation   is parabolic only in the normal direction.
As additional terms it  contains the time derivative of the middle-surface flexural strain.

Derivation of the model presents an extension of the results on the derivation of classical linear elastic  shells by Ciarlet and collaborators  to the poroelastic shells case. The new technical points include determination of the $2\times 2$ strain matrix, independent of the vertical direction, in the limit of the rescaled strains and identification of the pressure equation.  This term is not necessary to be determined in order to derive the classical flexural shell model.
\bigskip


{\bf AMS classcode}
MSC 35B25, $\quad$ MSC  74F10, $\quad$ MSC 74K25, $\quad$ MSC 74Q15, $\quad$ MSC 76S.

\end{abstract}

\section{Introduction}
A shell is a three dimensional body, defined by its middle surface $\mathcal{S}$ and a neighborhood of a small dimension (the thickness) along the normals to it. The shell is said to be thin when the thickness is much smaller then the minimum of its two radii of the curvature and of the characteristic length of the middle surface $\mathcal{S}$.

The basic engineering theory for the bending of thin shells is known as Kirchoff-Love theory or Love's first approximation. The equations were derived by the so called "direct method" (see \cite{Naghdi} and references therein) and not from the three dimensional equations. A derivation from the three dimensional, at the rigor of the continuum mechanics, is due to Novozhilov and we refer again to \cite{Naghdi} for both linear and nonlinear models.

A  different approach  to deriving  the shell equations is to suppose that the middle surface $\mathcal{S}$ is given as
$\mathcal{S}= \mathbf{X} ({\overline \omega})$, where $\omega \subset \mathbb{R}^2$ be an open bounded and simply connected set
with Lipschitz-continuous boundary $\partial \omega$ and $\mathbf{X} :{\overline \omega} \to \mathbb{R}^3 $ is a smooth injective immersion (that is $\mathbf{X}\in C^3$ and $3\times 2$ matrix $\nabla \mathbf{X}$ is of rank two).  The vectors $\mathbf{a}_\alpha(y) =\partial_\alpha \mathbf{X} (y)$, $\alpha=1,2$, are linearly independent for all $y\in {\overline \omega}$ and form the covariant basis of the tangent plane to the $2$-surface $\mathcal{S}$.

Then the reference configuration of the shell is of the form $\mathbf{r} ({\overline \Omega}^\ep )$, $\ep>0$, where
$\Omega^\ep = \o \times (-\ep/2 , \ep/2 )$ and
\begin{equation}\label{refconfig}
  \mathbf{r} =\mathbf{r} (y, x_3 ) = \mathbf{X} (y) + x_3  \mathbf{a}_3(y), \quad \mathbf{a}_3(y)=\frac{\mathbf{a}_1(y)\times\mathbf{a}_2(y)}{|\mathbf{a}_1(y)\times\mathbf{a}_2(y)|}.
\end{equation}
The associated equations of linearized three dimensional elasticity are then written in curvilinear coordinates with respect
to $(y_1 , y_2 , x_3 )\in {\overline \O}^\ep$. Their solutions represent the covariant components of the displacement field in the reference configuration $\mathbf{r} ({\overline \Omega}^\ep )$.

 Then Ciarlet and collaborators developed  the asymptotic analysis approach where the normal direction variable $x_3 \in (-\ep/2 , \ep/2 )$ was scaled by setting $y_3= x_3 / \ep $. This change of variables transforms the PDE to a singular perturbation problem in curvilinear coordinates on a fixed cylindrical domain. With such approach Ciarlet and collaborators have established the norm closeness between the  solution of the original three dimensional elasticity equations and the Kirchhoff-Love two dimensional flexural and membrane shell equations, in the limit as $\ep \to 0.$ For details, we refer  to the articles \cite{CiarletLods} and \cite{CiarletLodsMiara} and to the books \cite{CiarletDG} and \cite{Ciarlet3}. For the  complete asymptotic expansion we refer to the review paper \cite{DFY04} by Dauge et al.
Further generalizations to nonlinear  shells exist and were obtained using $\Gamma -$ convergence. We limit our discussion to the linear shells.

In the everyday life we  frequently meet shells (and other low dimensional bodies) which are saturated by a fluid. Many living tissues are fluid-saturated thin bodies like bones,  bladders, arteries and diaphragms and they are  interpreted as poroelastic  plates or shells.  Furthermore,  industrial filters are an example of poroelastic  plates and shells.

Our goal is to extend the above mentioned theory to the {\bf poroelastic} shells. In the case of the poroelastic plates, derivation of the mathematical model was undertaken in \cite{AMCAM:2012}.  As in the case of plates, these are the shells consisting of an elastic skeleton (the solid phase) and pores saturated by a viscous fluid (the fluid phase). Interaction between the two phases leads to an {\bf overall} or   {\bf effective} behavior described by the poroelasticity equations instead of the Navier elasticity equations, coupled with the mass conservation equation for the pressure field.  The equations form Biot's poroelasticity PDEs and can be found in \cite{Biot55}, \cite{Biot63} and in the selection  of Biot's   publications \cite{TOL}.

The effective  linear Biot's model corresponds to the homogenization of the complicated pore level fluid-structure interaction problem  based on
the continuum mechanics first principles, i.e. the Navier equations for the solid structure and by the Navier-Stokes equations for the flow. Small deformations are supposed and the interface between phases is linearized. The small parameter of the problem is  the ratio between characteristic pore size  and the domain size.
 If, in addition, we consider a periodic porous medium with connected fluid and solid phases, then  the two-scale poroelasticity equations can be obtained using  {\it formal two-scale expansions} in the small parameter, applied to the pore level fluid-structure equations.  For  details we refer  to the book \cite{SP-80}, the review \cite{Au-97} and the references therein.

 Convergence of the homogenization process for a given frequency was obtained in \cite{NG2}, using the {\it two-scale convergence technique}.   Convergence in space and time variables was proved by Mikeli\'c and  collaborators in  \cite{ClFGM} and \cite{FM:03}.  The upscaling result was presented in detail in \cite{AMCAM:2012} and we avoid to repeat it here. We point out that the upscaled model depends on the particular time scale known as {\it Terzaghi's time} $T_T$.
It is equal to the ratio between the viscosity times the characteristic length squared and the shear modulus of the solid structure multiplied by the permeability. If the characteristic time is much longer than $T_T$ than  flow dominates vibrations and the acceleration and memory effects can be neglected. The model is then  the quasi-static Biot model. For its derivation  from the first principles using homogenization techniques see  \cite{MW:10}.

For the direct continuum mechanics approach to Biot's equations, we refer to the monograph of Coussy \cite{Coussy}.
Using a direct approach, a model for a spherical poroelastic shell is proposed  in  \cite{TabPul96}.



In this paper we follow the approach of Ciarlet, Lods and Miara, as presented in the textbooks \cite{CiarletDG} and \cite{Ciarlet3}, and rigorously develop equations for a poroelastic flexural shell.

Successful recent approaches to the derivation of linear and nonlinear shell models
use the elastic energy functional. In our situation, presence of the flow makes the problem quasi-static, that  is time-dependent, and non-symmetric.
The  equations for the effective solid skeleton displacement  contain the pressure gradient and have the structure of a generalized Stokes system,
with the velocity field replaced by displacement. Mass conservation equation is parabolic in the pressure and contains the time derivative of the volumetric strain.

We recall that the quasi-static Biot system is well-posed only if there is a relationship between Biot parameters multiplying the pressure gradient in the displacement system and the time derivative of the divergence of the displacement in the pressure equation. We  were able to obtain in \cite{AMCAM:2012} the corresponding energy estimate. Similar estimates, for the equations in the curvilinear coordinates, will be obtained here.

In addition, there exists a major difference, with respect to the limit of the normalized $e_{33}$ term,
 compared to a classical derivation of the Kirchhoff-Love shell. In our poroelastic case, the limit
also contains the pressure field.
 Furthermore, the pressure oscillations persist and we prove the regularity and uniqueness for the limit problem. As in the poroelastic plate case,
it has a richer structure than the classical bending equation. We expect more complex  time behavior  in this model.


\section{Geometry of Shells and Setting of the Problem}
\setcounter{equation}{0}

We study the deformation and the flow in a poroelastic shell ${\tilde \Omega}_L^{\ell}=\mathbf{r} ({ \Omega}^\ell_L )$, $L, \ell>0$, where the injective mapping $\mathbf{r} $ is given by (\ref{refconfig}), for $x_3 \in (-\ell /2 , \ell /2 )$ and $( y_1 , y_2 )\in \omega_L$, diam $(\omega_L)=L$. We recall
the middle surface
$\mathcal{S}= \mathbf{X} ({\overline \omega}_L)$ is the image by a smooth injective immersion $\mathbf{X}$ of an open bounded and simply connected set  $\omega_L \subset \mathbb{R}^2$,
with Lipschitz-continuous boundary $\partial \omega_{L}$.
  We use the linearly independent vectors $\mathbf{a}_\alpha(y) =\partial_\alpha \mathbf{X} (y)$, $\alpha=1,2$, to form a covariant basis of the tangent plane to the $2$-surface $\mathcal{S}$.
The contravariant basis of the same plane is given
by the vectors $\mathbf{a}^\alpha(y)$ defined by
$
\mathbf{a}^\alpha(y) \cdot \mathbf{a}_\beta(y)= \delta^\alpha_\beta.
$
We extend these bases to the basis of the whole space $\mathbb{R}^3$ by
the vector $\mathbf{a}_3$ given  in (\ref{refconfig}) ($\mathbf{a}^3=\mathbf{a}_3$).
Now we collect the local contravariant and covariant bases into
the matrix functions
\begin{equation}\label{Q}
\mathbf{Q}= {\left[ \begin{array}{ccc} \mathbf{a}^1 & \mathbf{a}^2 & \mathbf{a}^3
\end{array}\right]}, \qquad \mathbf{Q}^{-1} = {\left[ \begin{array}{c} \mathbf{a}_1^T\\ \mathbf{a}_2^T\\
\mathbf{a}_3^T
\end{array}\right]}.
\end{equation}
The first fundamental form of the surface $\mathcal{S}$, or the metric tensor, in covariant
$\mathbf{A}_c=(a_{\alpha \beta})$ or contravariant $\mathbf{A}^c=(a^{\alpha
\beta})$ components  are given respectively by
$$
a_{\alpha\beta} = \mathbf{a}_\alpha\cdot\mathbf{a}_\beta,
\quad
a^{\alpha\beta} = \mathbf{a}^\alpha\cdot\mathbf{a}^\beta, \quad \alpha, \beta =1,2.
$$
Note here that because of continuity of $\mathbf{A}^c$ and compactness
of $ \overline{\omega}_L$, there are constants $M^c\geq m^c>0$ such that
\begin{equation}\label{mc}
m^c \mathbf{x} \cdot \mathbf{x} \leq \mathbf{A}^c(y) \mathbf{x} \cdot \mathbf{x} \leq M^c \mathbf{x} \cdot
\mathbf{x},
\qquad \mathbf{x} \in \mathbb{R}^3, y\in \overline{\omega}_{L}.
\end{equation}
These estimates, with different constants, hold for $\mathbf{A}_c$ as
well, as it is  the inverse of $\mathbf{A}^c$.
The second fundamental form of the surface  $\mathcal{S}$, also known as
the curvature tensor, in covariant $\mathbf{B}_c=(b_{\alpha \beta})$ or
mixed components $\mathbf{B} = (b^{\beta}_{\alpha})$ are given
respectively by
$$
b_{\alpha \beta} = \mathbf{a}^3 \cdot \partial_\beta \mathbf{a}_\alpha = -
\partial_\beta \mathbf{a}^3 \cdot \mathbf{a}_\alpha,
\quad
b^{\beta}_{\alpha}= \sum_{\kappa=1}^2 a^{\beta \kappa} b_{\kappa \alpha},\quad \alpha, \beta =1,2.
$$
The Christoffel symbols $\mathbf{\Gamma}^\kappa$ are defined by
$$
\Gamma^\kappa_{\alpha \beta}=\mathbf{a}^\kappa \cdot \partial_\beta
\mathbf{a}_\alpha=-\partial_\beta \mathbf{a}^\kappa \cdot \mathbf{a}_\alpha, \quad \alpha, \beta, \kappa =1,2.
$$
We will sometime use $\Gamma^3_{\alpha \beta}$ for $b_{\alpha
\beta}$. The area element along $\mathcal{S}$ is $\sqrt{a}dy$, where
$a:=\det{\mathbf{A}_c}$. By (\ref{mc}) it is uniformly positive, i.e.,
there is $m_a >0$ such that
\begin{equation}\label{a0}
0 < m_a \leq a(y), \qquad y \in \overline{\omega}_{L}.
\end{equation}
We also need the covariant derivatives $b^\kappa_\beta|_\alpha$
   which are defined by
\begin{equation}\label{symbb}
b^\kappa_\beta|_\alpha= \partial_\alpha b^\kappa_\beta
+\sum_{\tau=1}^2 (\Gamma^\kappa_{\alpha \tau} b^\tau_\beta - \Gamma^\tau_{\beta
\alpha} b^\kappa_\tau),
  \quad \alpha, \beta, \kappa =1,2.
\end{equation}
 In order to describe our results we also need the following differential operators:
\begin{gather}
\gamma_{\alpha\beta}(\vb) =    \frac{1}{2}(\partial_\alpha v_\beta + \partial_\beta v_\alpha) - \sum_{\kappa=1}^2 \Gamma^\kappa_{\alpha\beta} v_\kappa  -  b_{\alpha\beta} v_3, \qquad \alpha, \beta=1,2, \label{symbgam}\\
\rho_{\alpha\beta}(\vb) = \partial_{\alpha\beta} v_3 - \sum_{\kappa=1}^2 \Gamma^\kappa_{\alpha\beta} \partial_\kappa v_3 + \sum_{\kappa=1}^2 b^\kappa_\beta (\partial_\alpha v_\kappa - \sum_{\tau=1}^2\Gamma^\tau_{\alpha\kappa} v_\tau) + \sum_{\kappa=1}^2 b^\kappa_\alpha (\partial_\beta v_\kappa - \sum_{\tau=1}^2\Gamma^\tau_{\beta\kappa} v_\tau) \notag \\
 \quad + \sum_{\kappa=1}^2 b^\kappa_\alpha|_\beta v_\kappa - \sum_{\kappa=1}^2 b^\kappa_\alpha b_{\kappa\beta} v_3, \qquad \alpha, \beta =1,2,
\label{symbrho} \\
n_{\alpha\beta}|_{\beta} = 
 \partial_\beta n_{\alpha\beta} + \sum_{\kappa=1}^2\Gamma^\alpha_{\beta\kappa} n_{\beta\kappa} +\sum_{\kappa=1}^2 \Gamma^\beta_{\beta\kappa} n_{\alpha\kappa}, 
  \qquad \alpha , \beta =1,2,
  \notag \\
  n_{\alpha\beta}|_{\alpha\beta} = 
   \partial_\alpha (
   n_{\alpha\beta}|_\beta) + \sum_{\kappa=1}^2\Gamma^\kappa_{\alpha\kappa} 
    n_{\alpha\beta}|_\beta,  \qquad \alpha , \beta =1,2, \notag  
\end{gather}
   defined for smooth vector fields $\mathbf{v}$ and tensor fields $\nbb$.
\vskip4pt
The upper face (respectively lower face) of the shell  ${\tilde \Omega}^\ell_L$ is  ${\tilde \Sigma}^{\ell}_L =
\mathbf{r} ( \omega_L \times \{ x_3 =\ell /2 \}) =  \mathbf{r} (\Sigma^\ell_L )$
  (respectively ${\tilde \Sigma}^{-\ell}_L = \mathbf{r} ( \omega_L \times \{ x_3 =-\ell /2 \})= \mathbf{r} (\Sigma^{-\ell}_L )$. ${\tilde \Gamma}^\ell_L$ is the lateral boundary, ${\tilde \Gamma}^\ell_L = \mathbf{r} (\partial \omega_L \times (-\ell /2 , \ell /2 ))=\mathbf{r} (\Gamma^\ell_L) $.
We recall that the small parameter is the ratio between the shell thickness and the characteristic horizontal length is $\ep= \ell / L \ll 1$.

\begin{table}[ht]
 \caption{Parameter and unknowns  description}\label{Data}
 \begin{tabular}{|l|l|} \hline\hline
\emph{SYMBOL} & \emph{ QUANTITY }
  \\
\hline   $\mu$ & shear modulus  (Lam\'e's second parameter)   \\
\hline    $\lambda$ & Lam\'e's first parameter  \\
\hline    $\koefdva_G $  & inverse of Biot's modulus     \\
\hline    $\alpha$   & effective stress coefficient    \\
\hline    $k$  & permeability    \\
\hline    $\eta$  & viscosity    \\
\hline  $L$ and $\ell$  & midsurface  length and shell  width, respectively      \\
\hline    $\varepsilon = \ell  / L$  & small parameter      \\
\hline     $T= \eta L_c^2 / (k \mu)$   &  characteristic Terzaghi's time    \\
\hline    $U$   & characteristic displacement      \\
\hline    $P=U \mu/L$  & characteristic  fluid pressure  \\
 \hline   $\mathbf{u} = (u_1 , u_2 , u_3 )$ & solid phase displacement \\
 \hline   $p$ &  pressure \\
 \hline
\end{tabular}
\end{table}
\noindent
We note that Biot's diphasic equations describe behavior of the system at so called Terzaghi's time scale $T=\eta L^2_c / (k \mu ),$ where $L_c$ is the characteristic domain size, $\eta$ is dynamic viscosity, $k$ is permeability and $\mu$ is the
 shear modulus. For the list of all parameters see  Table \ref{Data}.

Similarly as in \cite{AMCAM:2012}, we chose as the characteristic length
 $L_c =\ell$, which leads  to the  Taber-Terzaghi  transversal time $T_{tab} = \eta \ell^2 / (k \mu)$.
 Another possibility was to choose the longitudinal time scaling with $T_{long} = \eta L^2 / (k \mu)$.
It would lead to different scaling in (\ref{Coweq2}) and the dimensionless permeability coefficient in (\ref{Bioteq2}) would not be $\ep^2$ but $1$. In the context of thermoelasticity, one has the same equations and Blanchard and Francfort rigorously derived in \cite{BlFr} the corresponding thermoelastic plate equations. We note that considering  the longitudinal time scale yields the effective model where the pressure (i.e. the temperature in thermoelasticity) is decoupled from the flexion.

Then
the quasi-static Biot equations  for the poroelastic body $ \tilde{\Omega}^{\ell}_L$
 take the following dimensional form:
\begin{gather}{\tilde \sigma} = 2\mu \mathbf{e} (\utb ) + ( \lambda \mbox{ div }\utb- \alpha {\tilde p} ) \Ibb  \; \mbox{ in } \; {\tilde \Omega}^\ell_L , \label{Coweq1} \\
-\mbox{ div } {\tilde \sigma} = -\mu \bigtriangleup \utb - (\lambda +\mu) \bigtriangledown \mbox{div }\utb +\alpha \bigtriangledown {\tilde p} =0 \; \mbox{ in } \; {\tilde \Omega}^\ell_L ,  \label{Coweq1a} \\
\frac{\partial }{ \partial { t}} (\koefdva_G {\tilde p} + \alpha \mbox{ div } \utb ) - \frac{ k}{\eta} \bigtriangleup {\tilde p} =0 \; \mbox{ in } \; {\tilde \Omega}^\ell_L .
\label{Coweq2}
\end{gather}
Note that $\mathbf{e}(\mathbf{u})=\sym\bigtriangledown \mathbf{u}$ and $\tilde \sigma$ is the stress tensor. All other quantities are defined in Table \ref{Data}.

We impose a given contact force ${\tilde \sigma} \nub = {\mathcal{\tilde P}}^{\pm \ell}_{L} $  and a given normal flux $ \displaystyle -\frac{k}{\eta}\frac{\p {\tilde p}}{\p x_3} =\tVred_{L}$
  at $x_3 = \pm \ell /2$. At the lateral boundary ${\tilde \Gamma}^\ell$ we impose a zero displacement and a zero normal flux.
Here $\nub$ is the outer unit normal at the boundary.  At initial time $t=0$ we prescribe the initial pressure $\pt^{\ell}_{L, \initial}$.

 Our goal is to extend the Kirchhoff-Love shell justification by Ciarlet, Lods et al and by Dauge et al to the poroelastic case.
\vskip5pt
We announce briefly the differential equations of the flexural poroelastic shell in dimensional form. Note that our mathematical result will be in the variational form and that differential form is only formal and written for reader's comfort.

\fbox{Effective dimensional equations:}
\vskip2pt


The model is given in terms of $\ub^\eff: \omega_L \to \ZR^3$ which is the vector of components of the displacement of the middle surface of the shell in the contravariant basis and $p^\eff : \Omega^\ell_L \to \ZR$ which is the pressure in the 3D shell.  Let us denote
the bending moment (contact couple) due to the variation in pore pressure across the plate
thickness by
\begin{equation}\label{moment}
\mbb = \frac{\ell^3}{12}  \tcC^c (\Abb^c \rhob(\ub^\eff)) \Abb^c +  \frac{2 \mu  \alpha}{\lambda +2\mu} \int_{-\ell/2}^{\ell/2} y_3 p^\eff dy_3  \Abb^c,
\end{equation}
where  ${\rhob} (\cdot )$ is given by  (\ref{symbrho}) and $\tcC^c$ is the elasticity tensor, usually appearing in the classical shell theories, given by
\begin{gather*}
\tcC^c \Ebb = 2\mu \frac{ \lambda}{\lambda+2\mu} \tr{(\Ebb)} \Ibb + 2\mu \Ebb, \qquad \Ebb \in  M^{3\times 3}_{\sym} .
\label{effcoef_dim}\end{gather*}
Then the model
in the differential formulation reads as follows:
\begin{equation}\label{main_druga_dim_4}
\aligned
&\sum_{\beta=1}^2\left(-(n_{\alpha\beta}+\sum_{\kappa=1}^2 b^\alpha_\kappa m_{\kappa\beta})|_{\beta} - \sum_{\kappa=1}^2 b_\kappa^\alpha (m_{\kappa\beta}|_{\beta}) \right)= (\cP_L^{+\ell})_\alpha+(\cP_L^{-\ell})_\alpha \mbox{ in } \omega_L,\quad \alpha=1,2,\\
&\sum_{\alpha,\beta=1}^2 \left( m_{\alpha\beta}|_{\alpha\beta} - \sum_{\kappa=1}^2 b^\kappa_\alpha b_{\kappa\beta} m_{\alpha\beta} - b_{\alpha\beta}n_{\alpha\beta} \right)= (\cP_L^{+\ell})_3+(\cP_L^{-\ell})_3 \mbox{ in } \omega_L,\\
&\frac{1}{2}(\partial_\alpha u_\beta^\eff + \partial_\beta u^\eff_\alpha) - \sum_{\kappa=1}^2 \Gamma^\kappa_{\alpha\beta} u_\kappa^\eff  -  b_{\alpha\beta} u_3^\eff = 0 \mbox{ in } \omega_L, \qquad \alpha,\beta=1,2,\\
&u_i^\eff=0, i=1,2,3, \quad \parder{u_3^\eff}{\nub}=0 \mbox{ on } \partial\omega_L, \qquad  \mbox{for every } \; t\in (0,T) ,
\endaligned
\end{equation}
\begin{equation}\label{main_treca_dim_diff}
\aligned
&\left(\beta_G + \frac{\alpha^2}{\lambda+2\mu}\right) \parder{p^\eff}{t}    -   \alpha \frac{2\mu}{\lambda+2\mu} \Abb^c  :   \rhob(\parder{\ub^\eff}{t}) y_3   -  \frac{k}{\eta}  \frac{\partial^2 p^\eff}{\partial (y_3)^2}    =0\\
& \quad \mbox{ in } \; (0,T) \times \omega_L \times (-\ell/2, \ell/2), \\
& \frac{k}{\eta} \parder{p^\eff}{y_3} = - \Vred_{L}, \quad \mbox{ on } (0,T) \; \times \omega_L \times (\{-\ell/2 \} \cup \{ \ell/2 \} ),\\
&p^\eff = p^{\ell}_{L, \initial} \quad \mbox{given  at } \; t=0.
\endaligned
\end{equation}
Here $(\cP_L^{\pm\ell})_i, i=1,2,3$ are components of the contact force ${\mathcal \Pt}_L^{\pm\ell}\circ \rb$ at $\Sigma^{\pm\ell}_L$ in the covariant basis, $\Vred_{L}=\tVred_{L}\circ \Xb$, $p^{\ell}_{L, \initial}=\pt^{\ell}_{L, \initial} \circ \rb$.
Thus, the poroelastic flexural shell model in the differential formulation is given for unknowns $\{ \nbb, \mbb, \ub^\eff, p^\eff \} $ and by equations (\ref{moment}), (\ref{main_druga_dim_4}) and (\ref{main_treca_dim_diff}). The components of $\nbb$ are the contact forces, linked to the constraint $\gamma (\mathbf{u}^\eff)=0$, and being the Lagrange multipliers in the problem. The components of $\mbb$ are the contact couples.
The first two equations in (\ref{main_druga_dim_4}) can be found in the differential equation of the Koiter shell model (see \cite[Theorem 7.1-3]{Ciarlet3}). The third equation is the restriction of approximate
   inextensibility of the shell. The first equation in (\ref{main_treca_dim_diff}) is the evolution equation for the effective pressure with associated boundary and initial conditions in the remaining part of (\ref{main_treca_dim_diff}).
 Note also that the same model holds for the shell clamped only on the portion of the boundary, i.e., the boundary condition in the fourth equation in (\ref{main_druga_dim_4}) holds for a subset of $\partial \omega_L$ with positive measure.


 In subsection  \ref{subDE} we present the dimensionless form of the problem.  In subsection \ref{secEU} we recall existence and uniqueness result of the smooth solution for the starting problem.  Subsection \ref{S4} is consecrated to the introduction of the problem in curvilinear coordinates and the rescaled problem, posed on the domain $\Omega =\o \times (-1/2,1/2)$.  In subsection \ref{subConv}  we formulate the main convergence results. In Section~\ref{S5} we study the a priori estimates for the family of solutions. Then in Section  \ref{S6} we study convergence of the solutions to the rescaled problem, as $\ep \to 0$.
 In Appendix 
 we give properties of the metric and curvature tensors.

 \section{Problem setting in curvilinear coordinates and the main results}
\setcounter{equation}{0}

\subsection{Dimensionless equations}\label{subDE}

We introduce the dimensionless unknowns and variable by setting
\begin{gather*}
    \koefdva = \koefdva_G \mu ; \quad  P = \frac{\mu U}{L} ; \quad U \utb^\ep = \utb ;
     \quad {  T= \frac{\eta \ell^2 }{ k \mu} ;} \quad {\tilde \lambda}= \frac{\lambda}{\mu};  \\
    P {\tilde p}^\ep = {\tilde p} ; \quad {\tilde y} L =y; \quad {\tilde x}_3 L=x_3 ; \quad \mathbf{\tilde r} L =\mathbf{r}; \quad \mathbf{\tilde X} L =\mathbf{ X} ; \quad {\tilde t} T= t ; \quad {\tilde \sigma}^\ep  \frac{\mu U}{L} = {\tilde \sigma}.
\end{gather*}
After dropping wiggles in the coordinates and in the time,
 the system (\ref{Coweq1})--(\ref{Coweq2}) becomes
\begin{eqnarray}
&&-\div {\tilde \sigma}^\ep = - \bigtriangleup \utb^\ep - {\tilde \lambda } \bigtriangledown \mbox{div }\utb^\ep +\alpha \bigtriangledown {\tilde p}^\ep =0 \; \mbox{ in }  (0,T) \times {\tilde \Omega}^\ep, \label{Bioteq1}
\\
&&{\tilde \sigma}^\ep = 2 \mathbf{e} (\utb^\ep ) + ( {\tilde \lambda} \mbox{ div }\utb^\ep - \alpha {\tilde p}^\ep ) \Ibb  \; \mbox{ in } \;  (0,T) \times {\tilde \Omega}^\ep, \label{Sig} \\
&&\frac{\partial }{ \partial t} (\koefdva {\tilde p}^\varepsilon + \alpha \mbox{ div } \utb^\varepsilon ) -{  \ep^2 } \triangle {\tilde p}^\varepsilon =0 \quad \mbox{ in } \;  (0,T) \times {\tilde \Omega}^ \varepsilon, \label{Bioteq2}
\end{eqnarray}
where  $\utb^\varepsilon=({\tilde u}_1^\varepsilon,{\tilde u}_2^\varepsilon,  {\tilde u}_3^\varepsilon)$ denotes the dimensionless displacement field and ${\tilde p}^\varepsilon $ the dimensionless pressure. We study a shell ${\tilde \Omega}^\ep$ with thickness $\varepsilon =\ell /L$ and section $\omega =\o_L /L$.
It is described by
\begin{equation*}
{\tilde \Omega}^\varepsilon = \mathbf{r} (\{ (x_{1}, x_{2}, x_{3}) \in \omega \times (-\varepsilon/2 , \varepsilon/2) \})= \mathbf{r} ({\Omega}^\varepsilon),
\end{equation*}
  ${\tilde \Sigma}^{\varepsilon}_+$ (respectively ${\tilde \Sigma}^{\varepsilon}_-$) is the upper face (respectively the lower face) of the shell ${\tilde \Omega}^ \varepsilon$. ${\tilde \Gamma}^ \varepsilon$ is the lateral boundary, ${\tilde \Gamma}^ \varepsilon = \partial \omega \times (-\varepsilon/2 , \varepsilon/2)$.

We suppose that a given dimensionless traction force is applied on ${\tilde \Sigma}^{\varepsilon}_+ \cup {\tilde \Sigma}^{\varepsilon}_-$
and impose  the shell is clamped on ${\tilde \Gamma}^ \varepsilon$:
\begin{gather}
{\tilde \sigma}^\varepsilon {\nub } = (2 \mathbf{e} (\utb^\varepsilon) -\alpha  {\tilde p}^\varepsilon I + {\tilde \lambda} (\mbox{div }\utb^\varepsilon) I) \nub=
  \ep^3 \mathcal{\Pt}_{\pm}   \;  \mbox{ on }\; {\tilde \Sigma}^{\varepsilon}_+ \cup {\tilde \Sigma}^{\varepsilon}_-, \label{FricBC}
\\
\mathbf{ \tilde u}^\varepsilon =0,  \quad \mbox{ on }\: {\tilde \Gamma}^\varepsilon. \label{FricBC1}
\end{gather}

For the pressure ${\tilde p}^\ep$,
at the lateral boundary ${\tilde \Gamma}^\varepsilon $ we
impose zero inflow/outflow flux:
\begin{equation}\label{Bclateral}
-  \bigtriangledown {\tilde p}^\varepsilon \cdot \nub = 0.
\end{equation}
and at ${\tilde \Sigma}^{\varepsilon }_+ \cup {\tilde \Sigma}^{\ep}_-$,  we set
\begin{gather}
-  \bigtriangledown {\tilde p}^\varepsilon  \cdot \nub = \pm \tVred.
\label{Bctop}
\end{gather}
Finally, we need an initial condition for ${\tilde p}^\ep$ at ${ t}=0$,
\begin{equation}\label{Bcbottom}
        {\tilde p}^\ep (x_1 , x_2 , x_3 , 0 ) = {\eps}\, {\tilde p}_{\initial} (x_1 , x_2 ) \quad \mbox{ in } \; {\tilde{\Omega}}^ \varepsilon.
\end{equation}

\def\sN2{\scriptscriptstyle i,j+1}
\def\seJ{{\scriptscriptstyle {i+\frac{1}{2},j}}}


Let $\mathcal{V}(\tilde{\Omega}^\varepsilon) = \{{\mathbf{\tilde{v}}} \in H^1(\tilde{\Omega}^\varepsilon; \mathbb{R}^3): \mathbf{\tilde{v}}|_{\tilde{\Gamma}^\varepsilon} = 0\}$.
Then the weak formulation corresponding to (\ref{Bioteq1})--(\ref{Bcbottom}) is given by

\vskip1pt
Find $\utb^{\ep}  \in H^1(0,T, \mathcal{V}(\tilde{\Omega}^\varepsilon) )$, ${\tilde p}^\ep \in H^1(0,T; H^1 ({\tilde \Omega}^{\ep}))$ such that it holds
\begin{eqnarray}
&&\int_{{\tilde \Omega}^\ep} 2  \ \mathbf{e}(\utb^\ep )  :  \mathbf{e}(\vtb)\ dx + {\tilde \lambda } \int_{{\tilde \Omega}^\ep} \mbox{ div } \utb^\ep  \mbox{ div }\vtb\ dx -\alpha\int_{{\tilde \Omega}^\ep}
{\tilde p}^\ep  \mbox{ div }\vtb \ dx\nonumber\\
 &&=\int_{{\tilde \Sigma}^{\ep}_+} \ep^3 \mathcal{\tilde P}_+ \cdot \vtb \ ds + \int_{{\tilde \Sigma}^{\ep}_-} \ep^3 \mathcal{\tilde P}_- \cdot \vtb \ ds, \quad \mbox{ for every } \; \vtb\in \mathcal{V}(\tilde{\Omega}^\varepsilon) \; \mbox{ and } t\in (0,T),\label{variational1} \\
&&\koefdva \int_{{\tilde \Omega}^\ep} \partial_t {\tilde p}^\ep {\tilde q} \ dx + \int_{{\tilde \Omega}^\ep} \alpha \mbox{ div } {\partial_t \utb^\ep }{\tilde q} \ dx +{  \ep^2 } \int_{{\tilde \Omega}^\ep}
\nabla {\tilde p}^\ep \cdot \nabla {\tilde q} \ dx\nonumber\\
&&=  \ep^2 \int_{{\tilde \Sigma}^{\ep}_-} \tVred  {\tilde q} \ ds -   \ep^2 \int_{{\tilde \Sigma}^{\ep}_+} \tVred  {\tilde q} \ ds, \quad \mbox{ for every } \; {\tilde q} \in H^1({\tilde \Omega}^\ep) \; \mbox{ and } t\in (0,T),\label{variational2}\\
&&{\tilde p}^\ep |_{\{t=0\}} ={\color{red}\eps}\, {\tilde p}_{\initial}, \quad \mbox{ in } {\tilde \Omega}^\ep.\label{variational3}
\end{eqnarray}


 Note that for two $3\times 3$ matrices $A$ and $B$ the Frobenius scalar product is denoted by $A:B= \displaystyle \tr{(A B^T)}$.

\subsection{Existence and uniqueness for the $\ep$-problem}\label{secEU}

In this subsection we recall the existence and uniqueness of a solution $\{ {\bf \tilde u}^\ep, {\tilde p}^\ep \}  \in$  $ H^{1}(0,T; \mathcal{V}(\tilde{\Omega}^\varepsilon) )\times H^{1}(0,T;H^1( {\tilde \Omega}^\ep ))$
  of the problem \eqref{variational1}-\eqref{variational3}. We follow \cite{AMCAM:2012} and get

\begin{proposition}\label{epexist}  Let us suppose
\begin{gather}\label{Hypoth}
     {\tilde p}_{\initial} \in H^2_0 ({\tilde \Omega}^\ep), \;
 \mathcal{P_{\pm}}\in H^{2}(0,T;  L^2 (\o ;  \mathbb{R}^3) )
   \; \mbox{and}   \;  \tVred \in H^{1} (0,T;L^2(\o )),
   \; \tVred |_{\{t=0\}} =0 .
\end{gather}
Then problem (\ref{variational1})--(\ref{variational3}) has a unique solution $\{ {\bf \tilde u}^\ep, {\tilde p}^\ep \}  \in H^1 (0,T;  \mathcal{V}(\tilde{\Omega}^\varepsilon) ) ) \times H^1 (0,T; H^1 ({\tilde \Omega^\eps)}).$
\end{proposition}


\subsection{Problem in Curvilinear Coordinates and the Scaled Problem}\label{S4}

Our goal is to find the limits of the solutions of problem
(\ref{variational1})--(\ref{variational3}) when $\varepsilon$ tends to zero. It is known from similar
considerations made for classical shells that asymptotic
behavior of the longitudinal and transverse displacements of the
elastic body is different. The same effect is expected in the
present setting. Therefore we need to consider asymptotic behavior
of the local components of the displacement $\utb^\varepsilon$. It can be
done in many ways, but in order to preserve some important
properties of bilinear forms, such as positive definiteness and
symmetry, we rewrite the equations in curvilinear coordinates
defined by $\mathbf{r}$. Then we formulate equivalent problems posed on
the domain independent of $\varepsilon$.

The covariant basis of the shell $\overline{\tilde{\Omega}^\varepsilon}$,
which is the three-dimensional manifold parameterized by $\mathbf{r}$, is
defined by
$$
\mathbf{g}^\varepsilon_i = \partial_i \mathbf{r} : \Omega^\varepsilon \rightarrow \mathbb{R}^3, \quad i=1, 2, 3.
$$
Vectors $\left\{ \mathbf{g}^{\varepsilon}_1, \mathbf{g}^{\varepsilon}_2, \mathbf{g}^{\varepsilon}_3\right\}$ are given by
\begin{align*}
   &\mathbf{g}^{\varepsilon}_1 = \mathbf{a}_1 (y) + x_3 \partial_{y_1}  \mathbf{a}_3 (y) , \\
   &\mathbf{g}^{\varepsilon}_2 = \mathbf{a}_2 (y) + x_3 \partial_{y_2}  \mathbf{a}_3 (y) , \\
   &\mathbf{g}^{\varepsilon}_3 = \mathbf{a}_3 (y).
\end{align*}

Vectors $\left\{ \mathbf{g}^{1,\varepsilon}, \mathbf{g}^{2,\varepsilon}, \mathbf{g}^{3,\varepsilon}\right\}$
satisfying
$$
\mathbf{g}^{j,\varepsilon} \cdot \mathbf{g}^\varepsilon_i = \delta_{ij} \mbox{ on } \overline{\Omega}^\varepsilon, \quad i, j =1, 2, 3,
$$
where $\delta_{ij}$ is the Kronecker symbol, form the contravariant
basis on $\overline{\tilde{\Omega}}^\varepsilon$. The contravariant
metric tensor $\mathbf{G}^{c,\varepsilon} = (g^{i j, \varepsilon})$, the covariant
metric tensor $\mathbf{G}^\varepsilon_c = (g^\varepsilon_{i j})$ and the Christoffel
symbols $\Gamma^{i, \varepsilon}_{j k}$ of the shell
$\overline{\tilde{\Omega}}^\varepsilon$ are defined by
$$
g^{i j, \varepsilon} = \mathbf{g}^{i,\varepsilon} \cdot \mathbf{g}^{j, \varepsilon}, \quad g^\varepsilon_{i
j} = \mathbf{g}^\varepsilon_i \cdot \mathbf{g}^\varepsilon_j,
\quad
\Gamma^{i, \varepsilon}_{j k} = \mathbf{g}^{i, \varepsilon} \cdot \partial_j \mathbf{g}^\varepsilon_k \mbox{ on } \overline{\Omega}^\varepsilon,
\quad i, j, k =1, 2, 3.
$$
We set
\begin{equation}\label{gammaind}
    {\bf \Gamma}^{i, \ep} =(\Gamma^{i, \ep}_{jk} )_{j,k=1, \dots ,3} \quad  \mbox{and} \quad {\tilde \gamb}_\ep (\mathbf{v}) = \frac{1}{2}(\nabla \mathbf{v} +\nabla \mathbf{v}^T)-\displaystyle \sum^3_{i=1} v_i\mathbf{\Gamma}^{i,\varepsilon}.
\end{equation}
Let $g^\varepsilon = \det{\mathbf{G}^\varepsilon_c}$.
 Until now we were using the canonical basis $\{ \mathbf{e}_1, \mathbf{e}_2, \mathbf{e}_3 \}$, for $\mathbb{R}^3$. Now  the displacement is rewritten in the
contravariant basis,
\[
\utb^\varepsilon \circ \mathbf{r} { (y_1 , y_2 , x_3 ) = \sum_{i=1}^3 {\tilde u}_i^{\varepsilon} \circ \rbb (y_1 , y_2 , x_3 ) \mathbf{e}_i = \sum_{i=1}^3 u^\varepsilon_i (y_1 , y_2 , x_3 ) \mathbf{g}^{i,\varepsilon}(y_1 , y_2 , x_3)},\quad
\vtb \circ \mathbf{r} = \sum_{i=1}^3 v_i \mathbf{g}^{i,\varepsilon},
\]
while for scalar fields we just change the coordinates 
\[
{\tilde p}^\varepsilon \circ \mathbf{r} = p^\varepsilon,\qquad
{\tilde q } \circ \mathbf{r} = q, \qquad \tVred \circ \mathbf{r} = \Vred,\qquad {\tilde p}_{\initial} \circ \mathbf{r} = p_{\initial},
\]
on $\overline{\Omega}^\varepsilon$.  The contact forces are rewritten in the covariant basis of the shell
$$
\mathcal{\tilde P}_\pm \circ \mathbf{r} =\sum_{i=1}^3 (\mathcal{P}_\pm)_i \mathbf{g}^\varepsilon_i \mbox{ on } \Sigma^\varepsilon_\pm.
$$
New vector functions are defined by
$$
\mathbf{u}^\varepsilon = u^\varepsilon_i \mathbf{e}_i, \quad
\mathbf{v} = v_i \mathbf{e}_i, \quad
\mathcal{P}_\pm = (\mathcal{P}_\pm)_i \mathbf{e}_i.
$$
 Note that $u^\varepsilon_i$ are not components of the physical displacement. They are just intermediate functions which will be used to reconstruct $\utb^\varepsilon$.
The corresponding function space to $\mathcal{V} (\tilde{\Omega}^\varepsilon)$ is
the space
$$
\mathcal{V} (\Omega^\varepsilon) = \{ \mathbf{v} \in H^1(\Omega^\varepsilon)^3 \: : \:
\mathbf{v}|_{\Gamma^\varepsilon} = 0 \}.
$$
Let $ \mathbf{Q}^\varepsilon = (\nabla
 \mathbf{r})^{-T}$ and let
\begin{equation}\label{symbcoef}
  \mathcal{C} \Ebb ={\tilde \lambda} (\tr{\Ebb}) \Ibb +2\Ebb , \quad \mbox{for all} \quad \Ebb\in  M^{3\times 3}_{\sym}.
\end{equation}
 Then the problem  (\ref{variational1})--(\ref{variational3}) can be written by
\begin{equation}\label{3dw1}
\aligned
&  \int_{\Omega^\varepsilon} \mathcal{C} \left(\mathbf{Q}^\varepsilon  {\tilde \gamb}_\ep  (\mathbf{u}^\varepsilon)(\mathbf{Q}^\varepsilon)^T \right)  :  \left(\mathbf{Q}^\varepsilon  {\tilde \gamb}_\ep (\mathbf{v})(\mathbf{Q}^\varepsilon)^T\right) \sqrt{g^\varepsilon} dy  -  \alpha \int_{\Omega^\varepsilon} p^\varepsilon \mbox{tr}{\left(\mathbf{Q}^\varepsilon  {\tilde \gamb}_\ep  (\mathbf{v})(\mathbf{Q}^\varepsilon)^T\right)} \sqrt{g^\varepsilon} dy \\
&\qquad =  \ep^3 \int_{\Sigma^\varepsilon_{+}}  \mathcal{P}_+ \cdot \mathbf{v} \sqrt{g^\varepsilon}ds + \ep^3 \int_{\Sigma^\varepsilon_{-}} \mathcal{P}_- \cdot \mathbf{v} \sqrt{g^\varepsilon} ds,
\qquad \vb \in \mathcal{V}(\Omega^\varepsilon), \mbox{ a.e. } t\in [0,T],\\
&  \int_{\Omega^\varepsilon} \koefdva \frac{\partial {p}^\varepsilon}{\partial t} q \sqrt{g^\varepsilon} dy +   \int_{\Omega^\varepsilon} \alpha \frac{\partial }{\partial t} \mbox{tr}{\left(\mathbf{Q}^\varepsilon {\tilde \gamb}_\ep  (\mathbf{u}^\varepsilon)(\mathbf{Q}^\varepsilon)^T \right)}  q \sqrt{g^\varepsilon} dy \\
 & \qquad +\varepsilon^2 \int_{\Omega^\varepsilon}  \mathbf{Q}^\varepsilon \nabla  p^\varepsilon  \cdot  \mathbf{Q}^\varepsilon \nabla  q  \sqrt{g^\varepsilon} dy
 = \ep^2  \int_{\Sigma^\varepsilon_{-}} \Vred q \sqrt{g^\varepsilon} ds -  \ep^2  \int_{\Sigma^\varepsilon_{+}} \Vred q \sqrt{g^\varepsilon} ds,\\
& \qquad q \in H^1(\Omega^\varepsilon), \mbox{ a.e. } t\in [0,T],\\
& p^\varepsilon = \eps\, p_{\initial}, \qquad \mbox{ for } t=0.
\endaligned
\end{equation}
 This is the problem in curvilinear coordinates.

Problems for all $\utb^\varepsilon, {\tilde p}^\varepsilon$ and $\mathbf{u}^\varepsilon,
p^\varepsilon$ are posed on $\varepsilon$--dependent domains.
In the sequel we follow the idea from Ciarlet, Destuynder
\cite{Ciarlet-Dest} and rewrite (\ref{3dw1}) on the canonical
domain independent of $\varepsilon$. As a consequence, the coefficients
of the resulting weak formulation will depend on $\varepsilon$
explicitly.

Let $\Omega = \omega \times (-1/2, 1/2)$ and let $\mathbf{R}^\varepsilon :
\overline{\Omega} \rightarrow \overline{\Omega}^\varepsilon$ be defined
by
$$
\mathbf{R}^\varepsilon (z)= (z^1, z^2, \varepsilon z^3), \quad
z \in \Omega, \ \varepsilon \in (0, \varepsilon_0).
$$
By $\Sigma_\pm = \omega \times \{\pm 1/2\}$ we denote the upper and
lower face of $\Omega$. Let $\Gamma = \partial \omega \times (-1/2, 1/2)$.
To the functions $\mathbf{u}^\varepsilon$, $p^\varepsilon$,  $g^\varepsilon$, $\mathbf{g}^\varepsilon_i$, $\mathbf{g}^{\varepsilon,i}$, $\mathbf{Q}^\varepsilon$,
$\Gamma^{i,\varepsilon}_{j k}$, $i, j, k = 1, 2, 3$ defined on
$\overline{\Omega}^\varepsilon$ we associate the functions $\mathbf{u} (\varepsilon)$,
$p(\varepsilon)$,  $g(\varepsilon)$, $\mathbf{g}_i
(\varepsilon)$, $\mathbf{g}^i (\varepsilon)$, $\mathbf{Q}(\varepsilon)$, $\Gamma^i_{i j} (\varepsilon)$,
$i, j, k = 1, 2, 3$ defined on $\overline{\Omega}$ by composition
with $\mathbf{R}^\varepsilon$. Let us also define
\begin{eqnarray*}
&&\mathcal{V}(\Omega) = \{ \mathbf{v}=(v_1, v_2, v_3) \in H^1(\Omega;\mathbb{R}^3) \: : \:
\mathbf{v}|_{\Gamma}= 0 \}.
\end{eqnarray*}
Then the problem (\ref{3dw1}) can be written  as
\begin{equation}\label{3dw2}
\aligned
&  \varepsilon\int_{\Omega} \mathcal{C} \left(\mathbf{Q}(\varepsilon)\gamb^\varepsilon (\mathbf{u}(\varepsilon))\mathbf{Q}(\varepsilon)^T \right)  :   \left(\mathbf{Q}(\varepsilon) \gamb^\varepsilon (\mathbf{v}) \mathbf{Q}(\varepsilon)^T\right) \sqrt{g(\varepsilon)} dz\\
& \qquad -  \varepsilon \alpha \int_{\Omega} p(\varepsilon) \mbox{tr}{\left(\mathbf{Q}(\varepsilon)\gamb^\varepsilon (\mathbf{v}) \mathbf{Q}(\varepsilon)^T\right)} \sqrt{g(\varepsilon)} dz\\
&= \ep^3  \int_{\Sigma_{+}}  \mathcal{P}_+ \cdot \mathbf{v} \sqrt{g(\varepsilon)}ds + \ep^3 \int_{\Sigma_{-}} \mathcal{P}_- \cdot \mathbf{v} \sqrt{g(\varepsilon)} ds, \qquad \mathbf{v} \in \mathcal{V}(\Omega), \mbox{ a.e. } t\in [0,T],\\
&  \varepsilon\int_{\Omega} \koefdva \frac{\partial p(\varepsilon)}{\partial t} q \sqrt{g(\varepsilon)} dz +   \varepsilon \int_{\Omega} \alpha \frac{\partial }{\partial t} \mbox{tr}{\left(\mathbf{Q}(\varepsilon) \gamb^\eps(\mathbf{u}(\varepsilon)) \mathbf{Q}(\varepsilon)^T \right)}  q \sqrt{g(\varepsilon)} dz \\
 &\qquad +\varepsilon^3 \int_{\Omega}  \mathbf{Q}(\varepsilon) \nabla^\varepsilon  p(\varepsilon)  \cdot \mathbf{Q}(\varepsilon) \nabla^\varepsilon  q   \sqrt{g(\varepsilon)} dz\\
 &
   =  \ep^2 \int_{\Sigma_{-}}  \Vred  q \sqrt{g(\varepsilon)} ds - \ep^2 \int_{\Sigma_{+}}  \Vred  q \sqrt{g(\varepsilon)} ds,
   \qquad q \in H^1(\Omega), \mbox{ a.e. } t\in [0,T],\\
& p(\varepsilon) = \eps\,p_{\initial} , \qquad \mbox{ for } t=0.
\endaligned
\end{equation}
Here
\begin{equation}\label{gammep}
  \gamb^\varepsilon (\mathbf{v})= \frac{1}{\varepsilon}\gamb_z(\mathbf{v}) + \gamb_y(\mathbf{v}) -  \sum_{i=1}^3 v_i\mathbf{\Gamma}^{i}(\varepsilon),
\end{equation}
\begin{equation*}
\aligned
\gamb_z (\mathbf{v}) &= {\left[
\begin{array}{ccc}
0& 0 &\frac{1}{2}\partial_3 v_1 \\
0& 0 &\frac{1}{2}\partial_3 v_2 \\
\frac{1}{2}\partial_3 v_1& \frac{1}{2}\partial_3 v_2 &\partial_3 v_3
\end{array}
\right]},\;
\gamb_y (\mathbf{v}) = {\left[
\begin{array}{ccc}
\partial_1 v_1 &  \frac{1}{2}(\partial_2 v_1 +\partial_1 v_2)& \frac{1}{2}\partial_1 v_3\\
\frac{1}{2}(\partial_2 v_1 +\partial_1 v_2) &  \partial_2 v_2 & \frac{1}{2}\partial_2 v_3\\
\frac{1}{2}\partial_1 v_3 &  \frac{1}{2} \partial_2 v_3 & 0
\end{array}
\right]},\\
\nabla^\varepsilon q &= \frac{1}{\varepsilon} \nabla_z q + \nabla_y q, \qquad \nabla_z q = {\left[\begin{array}{ccc} 0 & 0 & \partial_3 q \end{array}\right]}, \qquad \nabla_y q = {\left[\begin{array}{ccc} \partial_1 q & \partial_2 q & 0\end{array}\right]}.
\endaligned
\end{equation*}

We assume for simplicity that  $p_{\initial}=0$
and proceed with  asymptotic analysis.
 We start by  rescaling  the pressure
$$
\pi(\varepsilon)= \frac{p(\varepsilon)}{\varepsilon}, \qquad q = \frac{q}{\varepsilon}.
$$
The equations are now
\begin{equation}\label{3dw4}
\aligned
&  \varepsilon\int_{\Omega} \mathcal{C} \left(\mathbf{Q}(\varepsilon)\gamb^\varepsilon (\mathbf{u}(\varepsilon))\mathbf{Q}(\varepsilon)^T \right)  :   \left(\mathbf{Q}(\varepsilon) \gamb^\varepsilon (\mathbf{v}) \mathbf{Q}(\varepsilon)^T\right) \sqrt{g(\varepsilon)} dz\\
& \qquad -  \varepsilon^2 \alpha \int_{\Omega} \pi(\varepsilon) \mbox{tr}{\left(\mathbf{Q}(\varepsilon)\gamb^\varepsilon (\mathbf{v}) \mathbf{Q}(\varepsilon)^T\right)} \sqrt{g(\varepsilon)} dz \\
&= \varepsilon^3 \int_{\Sigma_{\pm}}  \mathcal{P}_\pm  \cdot \mathbf{v} \sqrt{g(\varepsilon)}ds,\qquad \mathbf{v} \in \mathcal{V}(\Omega), \mbox{ a.e. } t\in [0,T],\\
&  \varepsilon^3\int_{\Omega} \koefdva \frac{\partial \pi (\varepsilon)}{\partial t} q \sqrt{g(\varepsilon)} dz + \varepsilon^2  \int_{\Omega} \alpha \frac{\partial }{\partial t} \mbox{tr}{\left(\mathbf{Q}(\varepsilon) \gamb^\eps (\mathbf{u}(\varepsilon)) \mathbf{Q}(\varepsilon)^T \right)}  q \sqrt{g(\varepsilon)} dz\\
&\qquad +\varepsilon^5 \int_{\Omega}  \mathbf{Q}(\varepsilon) \nabla^\varepsilon  \pi(\varepsilon)  \cdot \mathbf{Q}(\varepsilon) \nabla^\varepsilon  q   \sqrt{g(\varepsilon)} dz = \mp\varepsilon^3 \int_{\Sigma_{\pm}}  \Vred  q \sqrt{g(\varepsilon)} ds,\\
&\qquad q \in H^1(\Omega), \; \mbox{a.e.} \; t\in [0,T], \\
& p(\varepsilon) = 0, \qquad \mbox{ for } t=0.
\endaligned
\end{equation}
Here and in the sequel we use the notation
$$
\aligned
&\mp \int_{\Sigma_{\pm}}  \Vred  q \sqrt{g(\varepsilon)} ds =  \int_{\Sigma_{-}}  \Vred  q \sqrt{g(\varepsilon)} ds -  \int_{\Sigma_{+}}  \Vred  q \sqrt{g(\varepsilon)} ds,\\
&\int_{\Sigma_{\pm}}  \mathcal{P}_\pm  \cdot \mathbf{v} \sqrt{g(\varepsilon)}ds= \int_{\Sigma_{+}}  \mathcal{P}_+  \cdot \mathbf{v} \sqrt{g(\varepsilon)}ds  + \int_{\Sigma_{-}} \mathcal{P}_-  \cdot \mathbf{v} \sqrt{g(\varepsilon)} ds.
\endaligned
$$
\begin{remark} \em
Existence and uniqueness of a smooth solution to problem (\ref{3dw4}) follows from Proposition \ref{epexist} and the smoothness of the curvilinear coordinates transformation.
\end{remark}

\subsection{Convergence results}\label{subConv}

In the remainder of the paper we make the following assumptions
\begin{assumptions}\label{Hypoth1}
  For simplicity, we assume that $p_{in} =0$, that $\Vred\in H^{1} (0,T;  L^2( \o))$,
   $V|_{\{ t=0\} } =0$ and
that  $\mathcal{P_{\pm}} \in H^{2}(0,T;   L^2( \o ; \mathbb{R}^3))$, with $ \mathcal{P_{\pm}} |_{\{ t=0\} } =0$.
\end{assumptions}
\noindent We recall that the differential operators $\gamb$ and $\rhob$
 are given by (\ref{symbgam}) and (\ref{symbrho}), respectively. Let
\begin{equation}\label{v}
\cV_F(\omega) = \{\vb\in H^1(\omega) \times H^1(\omega) \times H^2(\omega): \gamb(\vb)=0, \vb|_{\partial \omega} = 0,  \parder{v_3}{\nub}|_{\partial \omega} = 0\}.
\end{equation}
 We will suppose the classical hypothesis that leads to the "flexural shell" models:
\begin{equation}\label{flexhyp}
  \cV_F(\omega) \neq \{ 0 \}.
\end{equation}
 Let us formulate the boundary value problem in $\Omega=\o \times (-1/2 , 1/2)$ for the effective displacement and the effective pressure:
\vskip2pt
Find $\{\mathbf{u}, \pi^0  \} \in C([0,T];\cV_F(\omega) \times L^2(\O))$,  $\partial_{z_3} \pi^0 \in L^2 ((0,T)\times \Omega)$ satisfying the system
\begin{equation}\label{main_druga}
\aligned
&  \frac{\hnatri}{12} \int_\omega \tcC (\Abb^c\rhob(\ub))  :   \rhob(\vb)\Abb^c \sqrt{a} dz_1 dz_2 + \frac{2 \mut  \alpha}{\lambdat +2\mut}\int_\omega  \int_{-\hnaj/2}^{\hnaj/2} z_3 \pi^0 dz_3  \Abb^c  :   \rhob(\vb) \sqrt{a}dz_1 dz_2\\
 &\quad =   \int_{\o  } ( {\mathcal P}_{+} + {\mathcal P}_{-})  \cdot \vb \sqrt{a} dz_1 dz_2, \qquad \vb \in \cV_F(\omega).
\endaligned
\end{equation}
\begin{equation}\label{main_treca}
\aligned
&   \frac{d}{d t}\int_{\Omega} \left(\koefdva + \frac{\alpha^2 \mut }{\lambdat+2\mut}\right) \pi^0 q \sqrt{a} dz -  \int_{\Omega}  \frac{2\mut \alpha}{\lambdat+2\mut} \Abb^c  :   \rhob(\parder{\ub}{t}) z_3 q \sqrt{a} dz +  \int_{\Omega} \parder{\pi^0}{z_3} \parder{q}{z_3}   \sqrt{a} dz\\
& \qquad =  \mp \int_{\Sigma_{\pm}} \Vred  q \sqrt{a} \ dz_1 d z_2 \quad \mbox{in} \quad \mathcal{D}' (0,T), \qquad q \in H^1( -1/2 , 1/2; L^2 (\omega)),
\endaligned
\end{equation}
\begin{equation}\label{maincetvrta}
  \pi^0 = 0 \qquad \mbox{at} \quad t=0,
\end{equation}
where $\gamb (\cdot)$ and $\rhob (\cdot )$ are given by (\ref{symbgam}) and (\ref{symbrho}), respectively,
 and \begin{gather} \tcC \Ebb = 2\mut \frac{\lambdat}{\lambdat+2\mut} \tr{(\Ebb)} \Ibb + 2\mut \Ebb, \qquad \Ebb \in  M^{2\times 2}_{\sym} .\label{effcoef}\end{gather}

 \begin{proposition} Under  Assumption \ref{Hypoth1}, problem (\ref{main_druga})--(\ref{maincetvrta}) has a  unique solution $\{\mathbf{u}, \pi^0 \} \in C([0,T];\cV_F(\omega) \times L^2(\O))$, $\partial_{z_3} \pi^0 \in L^2 ((0,T)\times \Omega)$
     Furthermore, $\partial_t \pi^0_{3}\in L^2 ((0,T)\times \Omega)$ and $\partial_t \mathbf{u} \in L^2 (0,T; \cV_F(\omega) ))$.
\end{proposition}
\begin{proof} First we prove that $\{   \mathbf{u}, \pi^0 \} \in C([0,T];\cV_F(\omega) \times L^2(\O))$ and $\partial_{z_3} \pi^0 \in L^2 ((0,T)\times \Omega)$ imply a higher regularity in time:

Let us take $q= z_3 {\overline q} (z_1 , z_2)$, $ {\overline q} \in C^\infty ({\overline \omega})$ as a test function in (\ref{main_treca}). It yields 
\begin{equation}\label{comporeg}
  F= \left(\koefdva + \frac{\alpha^2  }{{\tilde \lambda}+2}\right) \int^{1/2}_{-1/2} z_3 \pi^0 \ dz_3 - \frac{1}{6}    \frac{\alpha }{ {\tilde \lambda}+2} \Abb^c  :  \rhob(\ub) \in H^1 (0,T; L^2 (\omega)).
\end{equation}
After inserting (\ref{comporeg}) into (\ref{main_druga}), it takes the form
\begin{equation}\label{main_drugaadd}
\aligned
&  \frac{\hnatri}{12} \int_\omega \tcC (\Abb^c\rhob(\ub))  :  \rhob(\vb)\Abb^c \sqrt{a} dz_1 dz_2 + \frac{ c_1^\beta}{3} \int_\omega  \Abb^c  : \rhob(\ub)   \Abb^c  :  \rhob(\vb) \sqrt{a}dz_1 dz_2\\
 &\quad =   \int_{\omega}  (\cP_++\cP_-)  \cdot \vb \sqrt{a} dz_1 dz_2 - c^\beta_2 \int_\omega F(t) \Abb^c  :  \rhob(\vb) \sqrt{a}dz_1 dz_2 , \qquad \vb \in \cV_F(\omega),
\endaligned
\end{equation}
with $ c^1_\beta =   \alpha^2/(  {\tilde \lambda}+2 )/ \big( \beta ( {\tilde \lambda}+2 ) +  \alpha^2 \big)$ and $c^2_\beta = 2   \alpha/(\koefdva (  {\tilde \lambda}+2 )+\alpha^2)$.
Taking the time derivative and using the time regularity of $F$ and ${\mathcal P}_{\pm} $, yields $\partial_t \ub \in L^2 (0,T;  \cV_F(\omega))$.
  For such $\ubb$ classical regularity theory for the second order linear parabolic equations applied at (\ref{main_drugaadd}) implies $\partial_t \pi^0 \in L^2 ((0,T)\times \Omega)$.

The existence and the uniqueness are based on the energy estimate. If we choose $\vb=\displaystyle \parder{{\mathbf{u}}}{t}$ as a test function for equation (\ref{main_druga}) and $\pi^0 $ as a test function in (\ref{main_treca}) and sum up the equations to obtain the equality
\begin{equation}\label{unique}
\aligned
&  \frac{1}{2}\frac{d}{dt} \bigg\{ \frac{\hnatri}{12}  \int_\omega \tcC (\Abb^c\rhob(\ub))  :  \rhob(\ub)\Abb^c \sqrt{a} dz_1 dz_2  +  \int_{\Omega} \left(\koefdva + \frac{\alpha^2 }{ {\tilde \lambda}+2}  \right)(\pi^0)^2 \sqrt{a} dz\\
&\qquad
  - 2  \int_{\o}  ({\mathcal P}_{+} + {\mathcal P}_{-})   \cdot \ub \sqrt{a} dz_1 dz_2  \bigg\}  +  \int_{\Omega} \left(\parder{\pi^0}{z_3}\right)^2  \sqrt{a} dz dt\\
  & \quad =  - \int_{\o} \partial_t     ({\mathcal P}_{+} + {\mathcal P}_{-})  \cdot \ub \sqrt{a} dz_1 dz_2   \mp \int_{\Sigma_{\pm}} \Vred  \pi^0 \sqrt{a} \ dz_1 d z_2.
\endaligned
\end{equation}
Equality (\ref{unique}) implies uniqueness of solutions to problem (\ref{main_druga})--(\ref{maincetvrta}). Concerning existence, equality (\ref{unique}) allows to obtain the uniform bounds for $\rhob (\ub )$ in $L^\infty (0,T; \cV_F(\omega) )$, for $\pi^0$ in $L^\infty (0,T; L^2 (\Omega) )$ and for $\partial_{z_3} \pi^0 $ in
     $L^2 (0,T; L^2 (\Omega) )$. Using \cite[Teorem 4.3-4.]{CiarletDG}
     and the classical weak compactness reasoning, we conclude the existence of at least one solution. \qed
\end{proof}

\begin{remark} \label{sepvariab} Let ${\overline \beta} = \beta + \displaystyle  \frac{\alpha^2}{ {\tilde \lambda} + 2} $. Using separation of variables, we obtain the formulas
\begin{gather} \pi^0  (t, z_1 , z_2 , z_3 ) =  -\Vred z_3 - \frac{4}{{\overline \beta} \pi^2}
\sum_{j=1}^{+\infty} \frac{(-1)^j}{(2j-1)^2 } \Big( \int^t_0 \exp \{ -\frac{\pi^2 (2j-1)^2 (t-\tau )}{ {\overline \beta}  } \} \frac{\partial }{\partial \tau} \big(\overline{\beta} \Vred ( \tau , z_1 , z_2 )  \notag \\  +\frac{2\alpha }{  {\tilde \lambda} + 2}
\Abb^c : \rhob (   \ub  ( \tau , z_1 , z_2 )  )\big) \ d\tau \Big) \sin ( (2j-1)\pi z_3 ),
\label{Sepvar}\\
  \int^{1/2}_{-1/2} z_3 \pi^0 \ d z_3 =-\frac{\Vred}{12}+
 \frac{8}{\pi^4 {\overline \beta}}\sum_{j=1}^{+\infty} \frac{1}{(2j-1)^4}  \int^t_0 \exp \{ -\frac{\pi^2 (2j-1)^2 (t-\tau )}{  {\overline \beta}  } \} \frac{\partial }{\partial \tau} \big(\overline{\beta} \Vred ( \tau , z_1 , z_2 ) \notag \\
 +  \frac{2\alpha  }{  {\tilde \lambda} + 2}
\Abb^c : \rhob (   \ub  ( \tau , z_1 , z_2 )  )\big) \ d\tau .
\label{Sepvar1}
\end{gather}
After plugging formula (\ref{Sepvar1}) into equation (\ref{main_druga}), we observe memory effects in the flexion equation.
\end{remark}

The main result of the paper is the following theorem.
\begin{theorem}\label{tmain}
Let us suppose Assumption \ref{Hypoth1}.
Let $\{ \ubb(\eps),\pi(\eps) \}\in H^1 (0,T;  \mathcal{V}(\Omega)) \times H^1 (0,T; H^1 (\Omega))$ be the unique solution of (\ref{3dw4})  and let $\{ \ub , \pi^0 \}$ be the unique solution for (\ref{main_druga})--(\ref{maincetvrta}). Then we obtain
$$
\aligned
&\ub(\eps) \to \ub \qquad \mbox{ strongly in } C([0,T]; H^1(\Omega;\ZR^3)),\\
&\frac{1}{\eps}\gamb^\eps (\ub(\eps)) \to \gamb^0 \qquad \mbox{ strongly in } C([0,T];L^2(\Omega;\ZR^{3\times 3})),\\
&\pi (\eps) \to \pi^0 \qquad \mbox{ strongly in } C([0,T];L^2(\Omega)),\\
&\parder{\pi(\eps)}{z_3} \to \parder{\pi^0}{z_3} \qquad \mbox{ strongly in } L^2(0,T;L^2(\Omega)),
\endaligned
$$
where 
\begin{gather}
\gamb^0 = \left[\begin{array}{cc}
-z_3 \rhob(\ub) &\begin{array}{c}
0\\0\end{array}\\
\begin{array}{cc}
0 & 0
\end{array}
&
\displaystyle \frac{\alpha }{{\tilde \lambda} + 2} \pi^0 + z_3 \frac{  {\tilde \lambda} }{ {\tilde \lambda} + 2 } \Abb^c  :  \rhob(\ub)
\end{array}\right].
\label{gamma0}\end{gather}
\end{theorem}
As a consequence of the convergence of the term $\frac{1}{\eps}\gamb^\eps(\ub(\eps))$, we obtain the convergence of the scaled stress tensor.

\begin{corollary}\label{cstress11}
For the stress tensor $\sigma(\eps) = \cC (\Qtbb(\eps) \gamb^\eps(\ub(\eps)) \Qtbb(\eps)^T) - \alpha p (\eps) \Ibb$ one has
\begin{equation}\label{cstress12}
\frac{1}{\eps} \sigma(\eps) \to \sigma = \cC (\Qtbb \gamb^0 \Qtbb^T)  - \alpha \pi^0 \Ibb \qquad \mbox{ strongly in } C([0,T];L^2(\Omega;\ZR^{3\times 3})).
\end{equation}
The limit stress in the local contravariant basis $\Qtbb=(\ab^1 \ \ab^2 \ \ab^3)$ is given by
$$
\Qtbb^T \sigma \Qtbb = \left[\begin{array}{cc}
\displaystyle -\frac{2 \alpha }{ {\tilde \lambda} + 2}  \pi^0 \Abb^c -  z_3 \left( \frac{2\mut  {\tilde \lambda} }{ {\tilde \lambda} + 2} (\Abb^c  :   \rhob(\ub)) \Ibb + 2 \Abb^c \rhob(\ub)\right) \Abb^c &0\\
0 & 0 \end{array}\right].
$$
\end{corollary}

\section{A priori estimates}\label{S5}
\setcounter{equation}{0}

Fundamental for a priori estimates for thin shell-like bodies is the following three-dimensional inequality of Korn's type for a family of linearly elastic shells.

\begin{theorem}[\mbox{\cite[Theorem 5.3-1]{Ciarlet3}}, \mbox{\cite[Theorem 4.1]{CiarletLodsMiara}}]\label{tKorn}
Assume that $ \mathbf{X} \in C^3(\overline{\omega}; \ZR^3)$. Then there exist constants $\eps_0>0, C>0$ such that for all $\eps\in(0,\eps_0)$ one has
$$
\|\vb\|_{H^1(\Omega;\ZR^3)} \leq \frac{C}{\eps} \|\gamb^\eps(\vb)\|_{L^2(\Omega;\ZR^{3\times 3})},  \forall \vb \in \cV (\Omega).
$$
\end{theorem}
\begin{remark}\em
Only a portion of the boundary with positive surface has to be clamped for the statement of the theorem to hold.
\end{remark}

Now we state the asymptotic properties of the
coefficients in the equation (\ref{3dw4}). Direct calculation shows
that there are constants $m_g$, $M_g$,
independent of $\eps \in (0, \eps_0)$, such that for all $z
\in \Omega$,
\begin{equation}\label{pozitiv}
m_g \leq \sqrt{g(\eps)} \leq M_g.
\end{equation}
The functions $\gb^i (\eps), \gb_i (\eps), g^{ij} (\eps), g(\eps),
\Qbb(\eps), \Gamma^i_{jk} (\eps)$ are in
$C(\overline{\Omega})$ by assumptions. Moreover, there is a
constant $C > 0$ such that for all $\eps \in (0, \eps_0)$,
\begin{eqnarray}
\nonumber
&&\| \gb^i (\eps) - \ab^i \|_\infty  +
\| \gb_i (\eps) - \ab_i \|_\infty \leq C \eps,\\
\label{ocjene}
&&     \|  \frac{\partial }{\partial z_3} \sqrt{g(\eps )} \|_\infty  + \| \sqrt{g(\eps)} - \sqrt{a} \|_\infty \leq C \eps,\\
\nonumber&&\| \Qbb (\eps) - \Qbb -\eps z_3 \Qbb^1\|_\infty \leq C \eps^2, \quad
\Qbb^1 = \left( \begin{array}{ccc}
 \sum_{\sigma=1}^2  b^1_\sigma \ab^\sigma &  \sum_{\sigma=1}^2 b^2_\sigma \ab^\sigma & 0
\end{array}\right),\\
&&\| \frac{1}{\eps} \left( \Gamma^{i}_{jk}(\eps) -
\Gamma^{i}_{jk}(0) \right) -z_3
\left(\frac{\partial}{\partial y_3}
 \Gamma^{i}_{jk}\right) (0) \|_\infty \leq C
\eps, \nonumber
\end{eqnarray}
where $\|\cdot \|_\infty$ is the norm in $C(\overline{\Omega})$.
For proof see \cite{CiarletLods}.
Additionally, in \cite[Theorem 3.3-1]{Ciarlet3}  the asymptotic behaviour of the Christoffel symbols  in $L^\infty$-norm is given by
\begin{equation}\label{CS}
\aligned
\Gamb^\kappa (\eps) &= \left[\begin{array}{ccc}
\Gamma^\kappa_{11} -\eps z_3 b^\kappa_1|_1  & \Gamma^\kappa_{12} -\eps z_3 b^\kappa_2|_1   & \displaystyle -b^\kappa_1-\eps z_3 \sum_{\tau=1}^2  b^\tau_1 b^\kappa_\tau  \\
\Gamma^\kappa_{21} -\eps z_3 b^\kappa_1|_2   & \Gamma^\kappa_{22} -\eps z_3 b^\kappa_2|_2 & \displaystyle -b^\kappa_2-\eps z_3 \sum_{\tau=1}^2  b^\tau_2 b^\kappa_\tau  \\
\displaystyle -b^\kappa_1-\eps z_3 \sum_{\tau=1}^2 b^\tau_1 b^\kappa_\tau   & \displaystyle-b^\kappa_2-\eps z_3 \sum_{\tau=1}^2   b^\tau_2 b^\kappa_\tau &0
\end{array} \right]+ O(\eps^2),
\endaligned
\end{equation}
where $\kappa=1,2$ and
\begin{equation}\label{CS2}
\aligned
\Gamb^3 (\eps) &= \left[\begin{array}{ccc}
b_{11} -\eps z_3 \displaystyle \sum_{\kappa=1}^2 b^\kappa_1 b_{\kappa 1}  & \displaystyle b_{12} -\eps z_3 \sum_{\kappa=1}^2 b^\kappa_1 b_{\kappa 2}   & 0  \\
\displaystyle b_{21} -\eps z_3 \sum_{\kappa=1}^2 b^\kappa_2 b_{\kappa 1}  & \displaystyle b_{22} -\eps z_3 \sum_{\kappa=1}^2 b^\kappa_2 b_{\kappa 2} & 0 \\
0   & 0 &0
\end{array} \right].
\endaligned
\end{equation}

In the following two lemmas we derive the a priori estimates in a classical way.
\begin{lemma}\label{lest1}
There is $C>0$ and $\eps_0>0$ such that for all $\eps \in ( 0, \eps_0 )$ one has
$$
\|\frac{1}{\eps}\gamb^\eps(\ub(\eps))\|_{L^\infty(0,T;L^2(\Omega;\ZR^{3\times 3}))},
\|\pi(\eps)\|_{L^\infty(0,T;L^2(\Omega;\ZR))},
\|\eps\nabla^\eps \pi(\eps)\|_{L^2(0,T;L^2(\Omega;\ZR^3))} \leq C.
$$
\end{lemma}
\begin{proof}
We set $\displaystyle \vb=\parder{\ub(\eps)}{t}$ and $q=\pi(\eps)$ in (\ref{3dw4}) and sum up the equations. After noticing that the pressure term from the first equation cancels with the compression term from the second equation we obtain
\begin{equation}\label{3dw3}
\aligned
& \frac{1}{2} \eps\frac{d}{dt}\int_{\Omega} \cC \left(\Qtbb(\eps)\gamb^\eps (\ub(\eps))\Qtbb(\eps)^T \right)  :  \left(\Qtbb(\eps) \gamb^\eps (\ub(\eps)) \Qtbb(\eps)^T\right) \sqrt{g(\eps)} dz\\
& \qquad
+ \frac{1}{2} \koefdva \eps^3 \frac{d}{dt}\int_{\Omega}  \pi(\eps)^2 \sqrt{g(\eps)} dz
+ \eps^5 \int_{\Omega}  \mathbf{Q}(\varepsilon) \nabla^\eps  \pi(\eps)  \cdot  \mathbf{Q}(\varepsilon) \nabla^\eps  \pi(\eps)   \sqrt{g(\eps)} dz\\
& =
 \eps^3 \int_{\Sigma_{\pm}}  {\mathcal P}_\pm  \cdot \parder{\ub(\eps)}{t} \sqrt{g(\eps)}ds
\mp \eps^3 \int_{\Sigma_{\pm}} \Vred  \pi(\eps) \sqrt{g(\eps)} ds.
\endaligned
\end{equation}
Dividing the equation by $\eps^3$ and
 using the product rule for derivatives
with respect to time on the right hand side we obtain
$$
\aligned
&  \frac{1}{2} \frac{d}{dt} \left(\frac{1}{\eps^2}\int_{\Omega} \cC \left(\Qtbb(\eps)\gamb^\eps (\ub(\eps))\Qtbb(\eps)^T \right)  :  \left(\Qtbb(\eps) \gamb^\eps (\ub(\eps)) \Qtbb(\eps)^T\right) \sqrt{g(\eps)} dz
+  \koefdva \int_{\Omega}  \pi(\eps)^2 \sqrt{g(\eps)} dz\right)\\
&\qquad + \eps^2 \int_{\Omega}  \mathbf{Q}(\varepsilon) \nabla^\eps  \pi(\eps)  \cdot  \mathbf{Q}(\varepsilon) \nabla^\eps  \pi(\eps)   \sqrt{g(\eps)} dz\\
& =
 \frac{d}{dt} \int_{\Sigma_{\pm}}  {\mathcal P}_\pm  \cdot\ub(\eps) \sqrt{g(\eps)}ds
- \int_{\Sigma_{\pm}}  \parder{{\mathcal P}_\pm }{t} \cdot \ub(\eps) \sqrt{g(\eps)}ds
\mp \int_{\Sigma_{\pm}} \Vred  \pi(\eps) \sqrt{g(\eps)} ds.
\endaligned
$$
Now we use the Newton-Leibniz formula for the right hand side terms and the notation
$$
{\mathcal P} = ({{\mathcal P}_++{\mathcal P}_-}) z_3 + \frac{{\mathcal P}_+-{\mathcal P}_-}{2}, \qquad \cV =  2 \Vred z_3
$$
to obtain
$$
\aligned
&  \frac{1}{2} \frac{d}{dt} \left(\frac{1}{\eps^2}\int_{\Omega} \cC \left(\Qtbb(\eps)\gamb^\eps (\ub(\eps))\Qtbb(\eps)^T \right)  :   \left(\Qtbb(\eps) \gamb^\eps (\ub(\eps)) \Qtbb(\eps)^T\right) \sqrt{g(\eps)} dz
+  \koefdva \int_{\Omega}  \pi(\eps)^2 \sqrt{g(\eps)} dz\right)\\
&\qquad + \eps^2 \int_{\Omega}  \mathbf{Q}(\varepsilon) \nabla^\eps  \pi(\eps)  \cdot  \mathbf{Q}(\varepsilon) \nabla^\eps  \pi(\eps)   \sqrt{g(\eps)} dz\\
& = \frac{d}{dt} \int_{\Omega}  \parder{}{z_3}({\mathcal P}  \cdot\ub(\eps) \sqrt{g(\eps)}) dz - \int_{\Omega} \parder{}{z_3} \left( \parder{{\mathcal P} }{t} \cdot \ub(\eps) \sqrt{g(\eps)} \right) dz\\
&\qquad - \int_{\Omega} \parder{}{z_3} \left( \cV  \pi(\eps) \sqrt{g(\eps)} \right) dz.
\endaligned
$$
Next we integrate this equality over time
\begin{equation}\label{3dw5}
\aligned
&  \frac{1}{2} \frac{1}{\eps^2}\int_{\Omega} \cC \left(\Qtbb(\eps)\gamb^\eps (\ub(\eps))\Qtbb(\eps)^T \right)  :   \left(\Qtbb(\eps) \gamb^\eps (\ub(\eps)) \Qtbb(\eps)^T\right) \sqrt{g(\eps)} dz
+  \frac{1}{2}\koefdva \int_{\Omega}  \pi(\eps)^2 \sqrt{g(\eps)} dz\\
&\qquad + \eps^2 \int_0^t \int_{\Omega}  \mathbf{Q}(\varepsilon) \nabla^\eps  \pi(\eps) \cdot  \mathbf{Q}(\varepsilon) \nabla^\eps  \pi(\eps)   \sqrt{g(\eps)} dz d\tau\\
& = \frac{1}{2} \left(\frac{1}{\eps^2}\int_{\Omega} \cC \left(\Qtbb(\eps)\gamb^\eps (\ub(\eps)|_{t=0})\Qtbb(\eps)^T \right)  :   \left(\Qtbb(\eps) \gamb^\eps (\ub(\eps)|_{t=0}) \Qtbb(\eps)^T\right) \sqrt{g(\eps)} dz
+  \koefdva \int_{\Omega}  \pi(\eps)^2|_{t=0} \sqrt{g(\eps)} dz\right)\\
&\qquad + \int_{\Omega}  \parder{}{z_3}({\mathcal P}  \cdot\ub(\eps) \sqrt{g(\eps)}) dz - \int_{\Omega}  \parder{}{z_3}({\mathcal P} |_{t=0} \cdot\ub(\eps)|_{t=0} \sqrt{g(\eps)}) dz\\
&\qquad - \int_0^t \int_{\Omega} \parder{}{z_3} \left( \parder{{\mathcal P} }{t} \cdot \ub(\eps) \sqrt{g(\eps)} \right) dz d\tau - \int_0^t \int_{\Omega} \parder{}{z_3} \left( \cV  \pi(\eps) \sqrt{g(\eps)} \right) dz d\tau.
\endaligned
\end{equation}
Since we have enough regularity for $\ub(\eps)$ we consider (\ref{3dw4}) for $t=0$. Then $\ub(\eps)|_{t=0}$  satisfies: for all $\vb \in \cV (\Omega)$
$$
\aligned
&  \frac{1}{\eps^2}\int_{\Omega} \cC \left(\Qtbb(\eps)\gamb^\eps (\ub(\eps)|_{t=0})\Qtbb(\eps)^T \right)  :  \left(\Qtbb(\eps) \gamb^\eps (\vb) \Qtbb(\eps)^T\right) \sqrt{g(\eps)} dz \\
& \qquad -  \frac{1}{\eps} \alpha \int_{\Omega} \pi(\eps)|_{t=0} \tr{\left(\Qtbb(\eps)\gamb^\eps (\vb) \Qtbb(\eps)^T\right)} \sqrt{g(\eps)} dz 
 =   \int_{\Sigma_{\pm}}  {\mathcal P}_\pm|_{t=0}  \cdot \vb \sqrt{g(\eps)}ds .
\endaligned
$$
Since the initial condition is $\pi(\eps)|_{t=0}=0$ this equation is a classical 3D equation of shell-like body in curvilinear coordinates rescaled on the canonical domain.  Next, $\displaystyle \mathcal{P}_{\pm} |_{t=0} =0$ and  the classical theory (see Ciarlet \cite{Ciarlet3}) yields $\displaystyle \ub (\eps) |_{t=0} =0$.
Using  Korn's inequality, positivity of $\cC$ and uniform positivity of $\Qbb(\eps)^T\Qbb(\eps)$ and $g(\eps)$ in (\ref{3dw5}) yields the estimate
$$
\aligned
&  \frac{1}{2} \frac{1}{\eps^2}\int_{\Omega} \cC \left(\Qtbb(\eps)\gamb^\eps (\ub(\eps))\Qtbb(\eps)^T \right)  : \left(\Qtbb(\eps) \gamb^\eps (\ub(\eps)) \Qtbb(\eps)^T\right) \sqrt{g(\eps)} dz
+  \frac{1}{2}\koefdva \int_{\Omega}  \pi(\eps)^2 \sqrt{g(\eps)} dz\\
&\qquad + \eps^2 \int_0^t \int_{\Omega}  \Qtbb(\eps) \nabla^\eps  \pi(\eps)  \cdot  \Qtbb(\eps) \nabla^\eps  \pi(\eps)   \sqrt{g(\eps)} dz d\tau \leq C.
\endaligned
$$
Since $\cC$ is positive definite and since $g(\eps)$ is uniformly positive definite (see \cite[Theorem 3.3-1]{Ciarlet3}) we obtain the following uniform bounds
$$
 \frac{1}{\ep} \|\Qtbb(\eps)\gamb^\eps (\ub(\eps))\Qtbb(\eps)^T \|_{L^\infty(0,T;L^2(\Omega; \ZR^{3\times 3}))},  \; \|\pi(\eps)\|_{L^\infty(0,T;L^2(\Omega))}, \; \|\eps \Qtbb(\eps)\nabla^\eps \pi (\eps)\|_{L^2(0,T;L^2(\Omega;\ZR^3))}.
$$
Since $\Qtbb(\eps)^T \Qtbb(\eps)$ is uniformly positive definite these estimates imply uniform bounds for
$$
 \frac{1}{\ep} \|\gamb^\eps (\ub(\eps))\Qtbb(\eps)^T \|_{L^\infty(0,T;L^2(\Omega; \ZR^{3\times 3}))}\quad
\mbox{ and }
\quad \|\eps \nabla^\eps \pi (\eps)\|_{L^2(0,T;L^2(\Omega;\ZR^3))}.
$$
Applying the uniform bounds for $\Qtbb(\eps)^T \Qtbb(\eps)$ once again implies the statement of the lemma.
\qed\end{proof}

\bigskip

We now first take the time derivative of the first equation in (\ref{3dw4}) and then insert $\displaystyle \vb=\parder{\ub(\eps)}{t}$ as a test function. Then we take $q=\parder{\pi(\eps)}{t}$ as a test function in the second equation in (\ref{3dw4}) and sum up the equations.
     The following equality holds
\begin{equation}\label{3dw6}
\aligned
&  \frac{1}{\eps^2}\int_0^T \int_{\Omega} \cC \left(\Qtbb(\eps)\gamb^\eps (\parder{\ub(\eps)}{t})\Qtbb(\eps)^T \right)  :  \left(\Qtbb(\eps) \gamb^\eps (\parder{\ub(\eps)}{t}) \Qtbb(\eps)^T\right) \sqrt{g(\eps)} dz d\tau\\
& \qquad
+ \koefdva  \int_0^T\int_{\Omega}  \parder{\pi(\eps)}{t} \parder{\pi(\eps)}{t} \sqrt{g(\eps)} dz d\tau
+ \frac{1}{2}\eps^2 \int_{\Omega}  \Qtbb(\eps) \nabla^\eps  \pi(\eps)  \cdot  \Qtbb(\eps) \nabla^\eps  \pi(\eps)  \sqrt{g(\eps)} dz\\
 &=    \int_0^T \int_{\Sigma_{\pm}}  \parder{{\mathcal P}_\pm}{t}  \cdot \parder{\ub(\eps)}{t} \sqrt{g(\eps)}ds d\tau
  \mp  \int_0^T \int_{\Sigma_{ \pm}} \Vred  \parder{\pi(\eps)}{t} \sqrt{g(\eps)} ds d\tau.
\endaligned
\end{equation}
Similarly as in Lemma~\ref{lest1} from this equality we obtain
\begin{lemma}\label{lest2}
There is $C>0$ and $\eps_0>0$ such that for all $\eps \in ( 0, \eps_0 )$ one has
$$
\|\frac{1}{\eps}\gamb^\eps(\parder{\ub(\eps)}{t})\|_{L^2(0,T;L^2(\Omega;\ZR^{3\times 3}))},
\|\parder{\pi(\eps)}{t}\|_{L^2(0,T;L^2(\Omega;\ZR))},
\|\eps\nabla^\eps \pi^\eps\|_{L^\infty(0,T;L^2(\Omega;\ZR^3))} \leq C.
$$
\end{lemma}

As a consequence of the scaled Korn's inequality from Theorem~\ref{tKorn} we obtain
\begin{corollary}\label{c1}  Let us suppose  Assumption \ref{Hypoth1} and let $\{ \mathbf{u} (\eps) , \pi (\eps) \} $ be the solution for problem (\ref{3dw4}). Then
there is $C>0$ and $\eps_0>0$ such that for all $\eps \in ( 0, \eps_0 ) $ one has
\begin{gather*}
   \|\frac{1}{\eps}\gamb^\eps(\ub(\eps))\|_{H^1(0,T;L^2(\Omega;\ZR^{3\times 3}))}, \quad
\|\ub(\eps)\|_{H^1(0,T;H^1(\Omega;\ZR^3))},\quad
\|\pi(\eps)\|_{H^1(0,T;L^2(\Omega;\ZR))},  \\
\|\parder{\pi(\eps)}{z_3}\|_{L^\infty(0,T;L^2(\Omega;\ZR))} \leq C.
\end{gather*}
 Furthermore,  there are $\ub \in H^1(0,T;H^1(\Omega;\ZR^3))$, $\pi^0 \in H^1(0,T;L^2(\Omega;\ZR))$ and $\gamb^0 \in $ \break $ H^1
 (0,T;L^2(\Omega;\ZR^{3 \times 3}))$ such that on a subsequence one has
\begin{equation}\label{konvergencije}
\aligned
&\ub(\eps) \rightharpoonup \ub \mbox{ weakly in } H^1(0,T;H^1(\Omega;\ZR^3)),\\
&\pi(\eps) \rightharpoonup \pi^0 \mbox{ weakly in } H^1(0,T;L^2(\Omega;\ZR)),\\
&\parder{\pi(\eps)}{z_3}  \rightharpoonup \parder{\pi^0}{z_3} \mbox{ weakly in } L^2(0,T;L^2(\Omega;\ZR)) \quad \mbox{and  weak * in } \quad L^\infty(0,T;L^2(\Omega;\ZR)),\\
&\frac{1}{\eps}\gamb^\eps(\ub(\eps)) \rightharpoonup \gamb^0 \mbox{ weakly in } H^1(0,T;L^2(\Omega;\ZR^{ 3 \times 3})).
\endaligned
\end{equation}
\end{corollary}
\begin{proof}
Straightforward.
\qed\end{proof}

\bigskip

Since $\frac{1}{\eps}\gamb^\eps(\ub(\eps))$ depends on $\ub(\eps)$ one expects that the limits $\ub$ and $\gamb^0$ are related. The following theorem gives the precise relationship.
It is fundamental for obtaining the limit model in the classical flexural shell derivation as well as in the present derivation, see \cite{CiarletLodsMiara}. The tensor $\gamb$ is the linearized change of metric tensor and $\rhob$ is linearized change of curvature tensor. They usually appear in shell theories as strain tensors.

\begin{theorem}[\mbox{\cite[Theorem 5.2-2]{Ciarlet3}, \cite[Lemma 3.3]{CiarletLodsMiara}}] \label{tglavni}
For any $\vb\in \cV(\Omega)$, let $\gamb^\eps (\vb) \in L^2(\Omega;\ZR^{3\times3})$  and let the tensors $\gamb(\vb), \rhob(\vb)$
belong to $L^2(\Omega;\ZR^{2\times2}), H^{-1} (\Omega; \ZR^{2\times2})$, respectively. Let the family $\{ \wbb (\eps)\}_{\eps>0} \subset \cV(\Omega)$ satisfies 
$$
\aligned
& {\wbb}(\eps) \rightharpoonup \wbb \mbox{ weakly in } H^1(\Omega;\ZR^{3}),\\
&\frac{1}{\eps}\gamb^\eps(\wbb(\eps)) \rightharpoonup {{\tilde \gamb}} \mbox{ weakly in } L^2(\Omega;\ZR^{3})
\endaligned
$$
as $\eps \to 0$. Then the limit function $\wbb $ is independent of transverse variable $z_3$, belongs to $H^1(\omega)\times H^1(\omega) \times H^2(\omega)$, satisfies the clamping boundary conditions
$$
\wbb |_{\partial \omega} = 0, \qquad \parder{w_3}{\nub}|_{\partial \omega} =0
$$
and the following  conditions
$$
\gamb(\wbb )=0, \qquad \rhob( \wbb ) \in L^2(\Omega; \ZR^{2\times2}) \mbox{ and } \parder{{\tilde \gamma }_{\alpha\beta}}{z_3} = - \rho_{\alpha\beta}(\wbb ).
$$

If in addition there is $\chib \in H^{-1}(\Omega;\ZR^{2\times2})$ such that as $\eps \to 0$
$$
\rhob(\wbb(\eps)) \to \chib \mbox{ strongly in } H^{-1}(\Omega;\ZR^{2\times2}),
$$
then
$$
\wbb (\eps) \to \wbb \mbox{ strongly in } H^1(\Omega;\ZR^3) \qquad \mbox{ and } \qquad \rhob(\wbb) = \chib \in L^2(\Omega;\ZR^{2\times 2}).
$$
\end{theorem}

\begin{remark}\label{r*}\em
The estimates from Lemma~\ref{lest1} and Lemma~\ref{lest2} yield uniform boundedness of $\ub(\eps)$ in $C^{0,1/2}([0,T],\cV(\Omega))$, $\frac{1}{\eps}\gamb^\eps(\ub(\eps))$ in $C^{0,1/2}([0,T]; L^2(\Omega;\ZR^{3\times 3})])$ and $\pi(\eps)$ in $C^{0,1/2}([0,T]; L^2(\Omega))$. Hence by Corollary~\ref{c1} { and Aubin-Lions lemma (see \cite{Simon}), there is a subsequence such that the $\{ \ub(\eps) \} $ converges to $\ub$ also in $C([0,T]; L^2 (\Omega;\ZR^3))$. \vskip0pt
Let $\varphi \in L^2 (\Omega)$. Then for every $\delta >0$, there exists $\varphi_\delta \in C^\infty_0 (\Omega)$ such that $\displaystyle \| \varphi -\varphi_\delta \|_{L^2 (\Omega )} \leq \delta$. Next
\begin{equation}
\label{convgg}
\aligned
  &\sup_{0\leq t\leq T} | \int_\Omega\frac{\partial }{\partial x_i} \big( \ub(\eps) (t) - \ub(t) \big) \varphi \ dx |\\
  &\quad \leq  \sup_{0\leq t\leq T} |\int_\Omega \frac{\partial }{\partial x_i} \big( \ub(\eps) (t) - \ub(t) \big) (\varphi -\varphi_\delta) \ dx | + \sup_{0\leq t\leq T} | \int_\Omega \frac{\partial \varphi_\delta }{\partial x_i} \big( \ub(\eps) (t) - \ub(t) \big)  \ dx |\\
&\quad  \leq C \delta \| \ub(\eps)  - \ub \|_{C([0,T]; H^1 (\Omega;\ZR^3))} + \| \frac{\partial \varphi_\delta }{\partial x_i} \|_{L^2 (\Omega)} \|\ub(\eps)  - \ub \|_{C([0,T]; L^2 (\Omega;\ZR^3))}\leq {\overline C} \delta,
 \endaligned
\end{equation}
for $\varepsilon \leq \varepsilon_0 (\delta)$. Therefore
$$ \lim_{\varepsilon \to 0} \sup_{0\leq t\leq T} | \int_\Omega \frac{\partial }{\partial x_i} \big( \ub(\eps) (t) - \ub(t) \big) \varphi \ dx | \leq {\overline C} \delta, $$
which yields
\begin{equation}\label{convggu}
  \ub(\eps)(t) \rightharpoonup \ub(t) \mbox{ weakly in } H^1(\Omega;\ZR^3) \qquad \mbox{for every}\quad t \in [0,T].
\end{equation}
\vskip0pt
Argument for the sequences $\displaystyle \{ \frac{1}{\eps}\gamb^\eps(\ub(\eps)) \}$ and $\{ \pi(\eps)\} $ is analogous and we get
}
$$
\aligned
&\pi(\eps)(t) \rightharpoonup \pi^0(t) \mbox{ weakly in } L^2(\Omega),\\
&\frac{1}{\eps}\gamb^\eps(\ub(\eps))(t) \rightharpoonup \gamb^0(t) \mbox{ weakly in } L^2(\Omega;\ZR^{ 3 \times 3})
\endaligned
$$
 for every $t \in [0,T]$.
\end{remark}

Thus we may apply Theorem~\ref{tglavni}, with $\wbb (\eps) = \mathbf{u} (\eps) (t)$, for each $t\in [0,T]$ and conclude that the limit points of $\{ \ub(\eps)(t) \} $ belong to $\cV_F(\omega)$.

Moreover we conclude that
$$
\gamma^0_{\alpha \beta} = \ogamma_{\alpha\beta} - z_3 \rho_{\alpha\beta}(\ub),
$$
where $\ogamma_{\alpha\beta}$ do not depend on $z_3$. We denote $\ogamb = [\ogamma_{\alpha\beta}]_{\alpha,\beta=1,2}$.


\section{Derivation of the limit model}\label{S6}
\setcounter{equation}{0}

In this section we derive a two-dimensional model. We obtain it in  five steps. In the first two we take the limit in (\ref{3dw4}) for special choices of test function. In this way we additionally specify  the limits $\ub, \pi^0, \gamb^0$ and the equations they satisfy. The part $\overline \gamb $ of $\gamb^0$, which is independent of $z_3$,
  is identified in Step 3 by techniques usually applied in the proof of strong convergence of strain tensors in the classic shell models derivations.  In Step 4 we prove the strong convergence of displacements, while in Step 5 we prove the strong convergence of stress tensors.

\textsc{Step 1} (Identification of ${\gamma}^0_{i3}$).
Now we are in a position to take the limit as $\eps \to 0$ in (\ref{3dw4}) with the first equation divided by $\eps$: 
$$
\aligned
&  \int_{\Omega} \cC \left(\Qtbb(\eps)\frac{1}{\eps}\gamb^\eps (\ub(\eps))\Qtbb(\eps)^T \right) {  :}  \left(\Qtbb(\eps)\eps \gamb^\eps (\vb) \Qtbb(\eps)^T\right) \sqrt{g(\eps)} dz\\
& \qquad -   \alpha \int_{\Omega} \pi(\eps) \tr{\left(\Qtbb(\eps)\eps \gamb^\eps (\vb) \Qtbb(\eps)^T\right)} \sqrt{g(\eps)} dz
=  \eps^2 \int_{\Sigma_{\pm}}  {\mathcal P}_\pm  \cdot \vb \sqrt{g(\eps)}ds,\\
&\qquad \vb \in \cV(\Omega), 
t\in [0,T].
\endaligned
$$
In the limit we obtain
$$
\aligned
& \int_{\Omega} \cC \left(\Qtbb(0) \gamb^0 \Qtbb(0)^T\right) {  :}   \left(\Qtbb(0) \gamb_z (\vb) \Qtbb(0)^T\right) \sqrt{g(0)} dz - \alpha \int_{\Omega} \pi^0 \tr{\left(\Qtbb(0)\gamb_z (\vb) \Qtbb(0)^T\right)} \sqrt{g(0)} dz  = 0,\\
& \qquad \vb \in \cV(\Omega), 
t\in [0,T],
\endaligned
$$
which,  using  $\Qbb(0) = \Qbb$  and $g(0)=a$, yields
$$
\aligned
& \int_{\Omega} \left( \cC \left(\Qtbb  \gamb^0 \Qtbb ^T\right) - \alpha \pi^0 \Ibb \right){  :}   \left(\Qtbb  \gamb_z (\vb) \Qtbb ^T\right) \sqrt{a} dz  = 0, \quad \vb \in \cV(\Omega), \mbox{ a.e. } t\in [0,T].
\endaligned
$$
From the definition of $\gamb_z$ and the function space $\cV(\Omega)$ we obtain
$$
\aligned
\left(\Qtbb^T\left( \cC \left(\Qtbb  \gamb^0 \Qtbb ^T\right) - \alpha \pi^0 \Ibb \right) \Qtbb\right)_{i3} = 0, \qquad i=1,2,3.
\endaligned
$$
This implies
$$
\aligned
\left( \left( \lambdat \tr{ \left(\Qtbb  \gamb^0 \Qtbb ^T\right)} - \alpha \pi^0 \right)\Qtbb^T \Qtbb+ 2\mut \Qtbb^T \Qtbb  \gamb^0 \Qtbb ^T \Qtbb\right)_{i3} = 0, \qquad i=1,2,3.
\endaligned
$$
Since
\begin{equation}\label{QTQ}
\Qtbb^T\Qtbb = \left[\begin{array}{cc}
\Abb^c & 0\\
0 & 1
\end{array}\right]
\end{equation}
we obtain expressions for the third column of $\gamb^0$ in terms of the remaining elements
\begin{equation}\label{tgamb3}
(\Qtbb^T \Qtbb \gamb^0)_{13}=(\Qtbb^T \Qtbb \gamb^0)_{23}=  \lambdat \tr{ \left( \Qtbb ^T \Qtbb  \gamb^0\right)} - \alpha \pi^0 + 2\mut\gamma^0_{33} = 0.
\end{equation}
The first two equations imply that
$$
\Abb^c \left[ \begin{array}{c}
\gamma^0_{13} \\ \gamma^0_{23}
\end{array}\right] =0
$$
and since $\Abb^c$ is positive definite we obtain that $\gamma^0_{13}=\gamma^0_{31}=\gamma^0_{23}=\gamma^0_{32}=0$.
From the third equation in (\ref{tgamb3}) we obtain
$$
\lambdat \Abb^c {  :}  \left[\begin{array}{cc}
\gamma^0_{11} & \gamma^0_{12}\\
\gamma^0_{12} & \gamma^0_{22}
\end{array}\right] - \alpha \pi^0  + (\lambdat +2\mut) \gamma^0_{33}=0.
$$
Thus we have obtained the following result.

\begin{lemma}\label{lstep1}
$$
\aligned
&\gamma^0_{13}=\gamma^0_{31}=\gamma^0_{23}=\gamma^0_{32}=0,\\
&\gamma^0_{33} = \frac{\alpha}{\lambdat+2\mut} \pi^0 -\frac{\lambdat}{\lambdat+2\mut} \Abb^c {  :}  \left[\begin{array}{cc}
\gamma^0_{11} & \gamma^0_{12}\\
\gamma^0_{12} & \gamma^0_{22}
\end{array}\right].
\endaligned
$$
\end{lemma}
From this lemma and Theorem~\ref{tglavni} we have that $\gamb^0$ is of the following form
$$
\gamb^0 = \left[\begin{array}{cc}
\ogamb-z_3 \rhob(\ub) &\begin{array}{c}
0\\0\end{array}\\
\begin{array}{cc}
0 & 0
\end{array}
&
\frac{\alpha}{\lambdat+2\mut} \pi^0 - \frac{\lambdat}{\lambdat+2\mut} \Abb^c {  :}  (\ogamb - z_3 \rhob(\ub))
\end{array}\right].
$$

\textsc{Step 2} (Taking the second limit).
Let $\vb \in \cV_F(\omega)$, where $\cV_F(\omega)$ is given in (\ref{v}), and let $\vb^1$ be given by
$$
\aligned
v^1_1 (z) & = - (\partial_1 v_3   + 2v_1 b^1_1 + 2v_2 b^2_1)z_3,\\
v^1_2 (z) & = - (\partial_2 v_3   + 2v_1 b^1_2 + 2v_2 b^2_2)z_3,\\
v^1_3 (z) & =  0.
\endaligned
$$
Then
$$
\gamb_z(\vb)=\gamb_z(\vb^1) + \gamb_y(\vb) - {  \sum_{i=1}^3 } v_i \Gamb^i(0) =0
$$
and $\vb(\eps) = \vb + \eps \vb^1 \in \cV(\Omega)$.


A simple calculation shows that $\frac{1}{\eps}\gamb^\eps(\vb(\eps)) = \Thetab (\vb)+ \eps \Gbb$, where
$$\Thetab(\vb)=\gamb_y(\vb^1) -  {  \sum_{i=1}^2 } v^1_i \Gamb^i(0) - {  \sum_{i=1}^3 } v_i (z_3
\left(\frac{\partial}{\partial y_3}
 \Gamb^{i}\right) (0))$$
 and $\Gbb$ is bounded in $L^\infty(\overline{\omega};\ZR^{3\times 3})$. For the test function $\vb(\eps)$  equation (\ref{3dw4}) now reads
$$
\aligned
&  \int_{\Omega} \cC \left(\Qtbb(\eps)\frac{1}{\eps}\gamb^\eps (\ub(\eps))\Qtbb(\eps)^T \right)  :  \left(\Qtbb(\eps) \frac{1}{\eps}\gamb^\eps (\vb(\eps)) \Qtbb(\eps)^T\right) \sqrt{g(\eps)} dz\\
& \qquad -   \alpha \int_{\Omega} \pi(\eps) \tr{\left(\Qtbb(\eps)\frac{1}{\eps}\gamb^\eps (\vb(\eps)) \Qtbb(\eps)^T\right)} \sqrt{g(\eps)} dz 
  =   \int_{\Sigma_{\pm}}  {\mathcal P}_\pm  \cdot \vb (\eps) \sqrt{g(\eps)}ds,\\
&\qquad \vb \in H^1_0(\omega;\ZR^3), 
t\in [0,T],\\
&   \int_{\Omega} \koefdva \parder{\pi(\eps)}{t} q \sqrt{g(\eps)} dz +   \int_{\Omega} \alpha \parder{}{t} \tr{\left(\Qtbb(\eps) \frac{1}{\eps} \gamb^\eps(\ub(\eps)) \Qtbb(\eps)^T \right)}  q \sqrt{g(\eps)} dz\\
&\qquad + \eps^2 \int_{\Omega} {  \Qtbb(\eps) }\nabla^\eps  \pi(\eps)  \cdot {  \Qtbb(\eps) } \nabla^\eps  q  \sqrt{g(\eps)} dz=   \mp\int_{\Sigma_{\pm} } \Vred  q \sqrt{g(\eps)} ds, \\ &\qquad q \in H^1(\Omega), \mbox{ a.e. }
t\in [0,T].
\endaligned
$$
In the limit when $\eps \to 0$ we obtain
\begin{equation}\label{3dw7}
\aligned
&  \int_{\Omega} \cC \left(\Qtbb \gamb^0   \Qtbb^T \right)  :  \left(\Qtbb  \Thetab(\vb) \Qtbb^T\right) \sqrt{a} dz -   \alpha \int_{\Omega} \pi^0 \tr{\left(\Qtbb \Thetab(\vb) \Qtbb^T\right)} \sqrt{a} dz\\
 &\qquad=   \int_{\omega}  (\cP_++\cP_-)  \cdot \vb \sqrt{a}ds ,
\qquad \vb \in H^1_0(\omega;\ZR^3), \mbox{ a.e. } t\in [0,T],\\
&   \int_{\Omega} \koefdva \parder{\pi^0}{t} q \sqrt{a} dz +   \int_{\Omega} \alpha \parder{}{t} \tr{\left(\Qtbb  \gamb^0 \Qtbb^T \right)}  q \sqrt{a} dz +  \int_{\Omega} \parder{\pi^0}{z_3} \Qtbb \eb_3  \cdot \parder{q}{z_3}  \Qtbb \eb_3  \sqrt{a} dz\\
& \qquad =  \mp\int_{\Sigma_{\pm}}  \Vred  q \sqrt{a} ds, \qquad q \in H^1(\Omega), 
t\in [0,T].
\endaligned
\end{equation}
Note that $\Qtbb \eb_3  \cdot \Qtbb \eb_3  = 1$.

According to Lemma~\ref{lTheta} (in the Appendix) one has
$$
\Thetab(\vb) = - z_3  \left[\begin{array}{cc}
\rhob(\vb)   & \begin{array}{c}\displaystyle \sum_{\kappa=1}^2 b^\kappa_1 (\partial_\kappa v_3 +  \sum_{\tau=1}^2 v_\tau b^\tau_\kappa) \\ \displaystyle  \sum_{\kappa=1}^2 b^\kappa_2 (\partial_\kappa v_3 + \sum_{\tau=1}^2 v_\tau b^\tau_\kappa)
\end{array}\\
\begin{array}{cc}
\displaystyle \sum_{\kappa=1}^2 b^\kappa_1 ( \partial_\kappa v_3 +  \sum_{\tau=1}^2 v_\tau b^\tau_\kappa ) & \displaystyle  \sum_{\kappa=1}^2 b^\kappa_2 (\partial_\kappa v_3 + \sum_{\tau=1}^2 v_\tau b^\tau_\kappa)
\end{array} & 0
\end{array}\right].
$$
Thus, using (\ref{QTQ}) we obtain
\begin{equation}\label{trTheta}
\aligned
\tr{(\Qtbb \Thetab(\vb) \Qtbb^T)} &= \tr{(\Qtbb^T \Qtbb \Thetab(\vb))} =
- z_3 \Abb^c {  :}  \rhob(\vb).
\endaligned
\end{equation}
Next, using Lemma~\ref{lstep1} {  and Remark \ref{r*}}, we compute
\begin{equation}\label{trtgama}
\aligned
\tr{(\Qtbb \gamb^0 \Qtbb^T)} &= \tr{(\Qtbb^T \Qtbb \gamb^0)} = \Abb^c {  :}  \left[\begin{array}{cc}
\gamma^0_{11} & \gamma^0_{12}\\
\gamma^0_{12} & \gamma^0_{22}
\end{array}\right] + \gamma^0_{33}\\
&=
\Abb^c {  :}  \left[\begin{array}{cc}
\gamma^0_{11} & \gamma^0_{12}\\
\gamma^0_{12} & \gamma^0_{22}
\end{array}\right] +
\frac{\alpha}{\lambdat+2\mut} \pi^0 -\frac{\lambdat}{\lambdat+2\mut} \Abb^c {  :}  \left[\begin{array}{cc}
\gamma^0_{11} & \gamma^0_{12}\\
\gamma^0_{12} & \gamma^0_{22}
\end{array}\right]\\
&=
\frac{2\mut}{\lambdat+2\mut} \Abb^c {  :}  \left[\begin{array}{cc}
\gamma^0_{11} & \gamma^0_{12}\\
\gamma^0_{12} & \gamma^0_{22}
\end{array}\right] +
\frac{\alpha}{\lambdat+2\mut} \pi^0\\
&=
\frac{2\mut}{\lambdat+2\mut} \Abb^c {  :}  \left[\begin{array}{cc}
\ogamma_{11} & \ogamma_{12}\\
\ogamma_{12} & \ogamma_{22}
\end{array}\right] - z_3 \frac{2\mut}{\lambdat+2\mut} \Abb^c {  :}  \rhob(\ub) +
\frac{\alpha}{\lambdat+2\mut} \pi^0.
\endaligned
\end{equation}
Further, using (\ref{QTQ}) and Lemma~\ref{lstep1} we obtain
\begin{equation}\label{gama}
\aligned
\Qtbb^T \Qtbb \gamb^0 \Qtbb^T \Qtbb &= \left[\begin{array}{cc}
\Abb^c (\ogamb - z_3 \rhob(\ub)) \Abb^c & 0 \\
0 & \gamma^0_{33}
\end{array}\right]\\
&= \left[\begin{array}{cc}
\Abb^c (\ogamb - z_3 \rhob(\ub)) \Abb^c & 0 \\
0 & \frac{\alpha}{\lambdat+2\mut} \pi^0 -\frac{\lambdat}{\lambdat+2\mut} \Abb^c {  :}  (\ogamb - z_3 \rhob(\ub))
\end{array}\right].
\endaligned
\end{equation}
The main elastic term is now
$$
\aligned
&\hspace{-2ex}  \int_{\Omega} \cC \left(\Qtbb \gamb^0   \Qtbb^T \right) {  :}  \left(\Qtbb  \Thetab(\vb) \Qtbb^T\right) \sqrt{a} dz\\
&= \int_{\Omega} \lambdat \tr{\left(\Qtbb \gamb^0   \Qtbb^T \right)}  \tr{\left(\Qtbb \Thetab(\vb)   \Qtbb^T \right)} + 2 \mut \Qtbb \gamb^0   \Qtbb^T {  :}  \Qtbb  \Thetab(\vb) \Qtbb^T \sqrt{a} dz\\
&= \int_{\Omega} \lambdat \left( \frac{2\mut}{\lambdat+2\mut} \Abb^c {  :}  \ogamb - z_3 \frac{2\mut}{\lambdat+2\mut} \Abb^c {  :}  \rhob(\ub) +
\frac{\alpha}{\lambdat+2\mut} \pi^0 \right)  \left(
- z_3 \Abb^c {  :}  \rhob(\vb)\right) \sqrt{a}dz\\
&\quad + \int_{\Omega} 2 \mut \Qtbb^T \Qtbb \gamb^0  \Qtbb^T \Qtbb  {  :}    \Thetab(\vb)  \sqrt{a} dz\\
&= \int_{\Omega}
(z_3)^2 \frac{2\lambdat \mut}{\lambdat+2\mut} \Abb^c {  :}  \rhob(\ub) \Abb^c {  :}  \rhob(\vb)
+
\frac{\lambdat  \alpha}{\lambdat+2\mut} \pi^0   \left(
- z_3 \Abb^c {  :}  \rhob(\vb)\right) \sqrt{a}dz\\
&\quad + \int_{\Omega} 2 \mut
(z_3)^2 \Abb^c  \rhob(\ub) \Abb^c   {  :}    \rhob(\vb)
\sqrt{a} dz.
\endaligned
$$
Using the tensor  $ \tcC$, defined by (\ref{effcoef}),  
the elastic term can now be written by 
\begin{equation}\label{mainelastic}
\aligned
&\hspace{-2ex}  \int_{\Omega} \cC \left(\Qtbb \gamb^0   \Qtbb^T \right) {  :}  \left(\Qtbb  \Thetab(\vb) \Qtbb^T\right) \sqrt{a} dz\\
&= 
\frac{\hnatri}{12} \int_\omega \tcC (\Abb^c\rhob(\ub))  :  \rhob(\vb)\Abb^c \sqrt{a} dz_1 dz_2
+ \int_\Omega
\frac{\lambdat  \alpha}{\lambdat+2\mut} \pi^0   \left(
- z_3 \Abb^c {  :}  \rhob(\vb)\right) \sqrt{a}dz.
\endaligned
\end{equation}

The first equation from (\ref{3dw7}) now becomes: for
all $\vb \in \cV_F(\omega)$ one has
$$
\aligned
&
\frac{\hnatri}{12} \int_\omega \tcC (\Abb^c\rhob(\ub))  :  \rhob(\vb)\Abb^c \sqrt{a} dz_1 dz_2
+ \int_\Omega
\frac{\lambdat  \alpha}{\lambdat+2\mut} \pi^0   \left(
- z_3 \Abb^c  :  \rhob(\vb)\right) \sqrt{a}dz\\
&\qquad -   \alpha \int_{\Omega} \pi^0 \left(
- z_3 \Abb^c  :  \rhob(\vb)\right) \sqrt{a} dz
=   \int_{\omega}  ({\mathcal P}_+ + {\mathcal P}_-)  \cdot \vb \sqrt{a} dz_1 dz_2.
\endaligned
$$
This implies
\begin{equation}\label{druga}
\aligned
&  \frac{\hnatri}{12} \int_\omega \tcC (\Abb^c\rhob(\ub)) {  :}  \rhob(\vb)\Abb^c \sqrt{a} dz_1 dz_2 + \frac{2 \mut  \alpha}{\lambdat+2\mut}\int_\omega  {  \bigg(} \int_{-\hnaj/2}^{\hnaj/2} z_3 \pi^0 dz_3 {  \bigg)} \Abb^c {  :}  \rhob(\vb) \sqrt{a}dz_1 dz_2\\
 &\quad =   \int_{\omega}  ({\mathcal P}_+ + {\mathcal P}_-)  \cdot \vb \sqrt{a} dz_1 dz_2, \qquad \vb \in \cV_F({  \omega}).
\endaligned
\end{equation}
The matrix $\ogamb$ does not appear in (\ref{druga}) and is not important in the classical shell theory. However this is not the case here since the term $\ogamb$ appears in the second equation in (\ref{3dw7}). It will turn out to be 0 in the proof of the strong convergence, see Step 3. The equation (\ref{druga}) appears in the classical flexural shell model without the pressure $\pi^0$ term.

The second equation in (\ref{3dw7}) can be now written {as}
$$
\aligned
&   \int_{\Omega} \koefdva \parder{\pi^0}{t} q \sqrt{a} dz +   \int_{\Omega} \alpha \parder{}{t} \left( \frac{2\mut}{\lambdat+2\mut} \Abb^c  :  \ogamb - z_3 \frac{2\mut}{\lambdat+2\mut} \Abb^c  :  \rhob(\ub) +
\frac{\alpha}{\lambdat+2\mut} \pi^0\right)  q \sqrt{a} dz\\
&\qquad +  \int_{\Omega} \parder{\pi^0}{z_3} \parder{q}{z_3}   \sqrt{a} dz =  \mp \int_{\Sigma_{\pm}} \Vred  q \sqrt{a} ds, \qquad q \in H^1(\Omega).
\endaligned
$$
Then
\begin{equation}\label{treca}
\aligned
&   \int_{\Omega} \left(\koefdva + \frac{\alpha^2}{\lambdat+2\mut}\right)\parder{\pi^0}{t} q \sqrt{a} dz +   \int_{\Omega} \alpha \parder{}{t} \left( \frac{2\mut}{\lambdat+2\mut} \Abb^c {  :}  \ogamb - z_3 \frac{2\mut}{\lambdat+2\mut} \Abb^c {  :}  \rhob(\ub)\right)  q \sqrt{a} dz\\
&\qquad +  \int_{\Omega} \parder{\pi^0}{z_3} \parder{q}{z_3}   \sqrt{a} dz
=  \mp \int_{\Sigma_{\pm}} \Vred  q \sqrt{a} ds, \qquad q \in H^1(\Omega).
\endaligned
\end{equation}
A flexural poroelastic shell model will follow from 
(\ref{druga}), (\ref{treca}) once $\ogamb$ is determined.
\vskip0pt
 \begin{remark}\label{otherpress} Let us set
$$P^G = \pi^0 +  \frac{2\alpha}{\beta ({\tilde \lambda} +2)+\alpha^2} \Abb^c  :  \ogamb  $$
Then the couple $\{\ub , P^G \}$ satisfies the system
\begin{gather}
 \frac{\hnatri}{12} \int_\omega \tcC (\Abb^c\rhob(\ub))  : \rhob(\vb)\Abb^c \sqrt{a} dz_1 dz_2 + \frac{2 \mut  \alpha}{\lambdat+2\mut}\int_\omega   \bigg( \int_{-\hnaj/2}^{\hnaj/2} z_3 P^G dz_3  \bigg) \Abb^c  :  \rhob(\vb) \sqrt{a}dz_1 dz_2 \notag \\
  =   \int_{\omega}  ({\mathcal P}_+ + {\mathcal P}_-)  \cdot \vb \sqrt{a} dz_1 dz_2, \qquad \vb \in \cV_F( \omega). \label{druga1} \\
 \int_{\Omega} \left(\koefdva + \frac{\alpha^2}{\lambdat+2\mut}\right)\parder{P^G}{t} q \sqrt{a} dz -   \int_{\Omega} \alpha  z_3 \frac{2\mut}{\lambdat+2\mut} \Abb^c  :  \parder{}{t}  \rhob(\ub)  q \sqrt{a} dz \notag \\
 +  \int_{\Omega} \parder{P^G}{z_3} \parder{q}{z_3}   \sqrt{a} dz
=  \mp \int_{\Sigma_{\pm}} \Vred  q \sqrt{a} ds, \qquad q \in H^1(\Omega).\label{treca1}
\end{gather}
Analogously to Section \ref{subConv}, we prove that system (\ref{druga1})--(\ref{treca1}) has a unique solution. It yields  convergence of the whole sequence $\{\ub(\eps) \}$. Unfortunately, it is still not enough to have conclusions for $\{ \pi (\eps) \}$.
  \end{remark}

\textsc{Step 3} (Identification of $\ogamb$ and the strong convergence of the strain tensor and the pressure).
We start with
$$
\aligned
\Lambda (\eps)(t) &= \frac{1}{2}\int_{\Omega} \cC \left(\Qtbb(\eps)\left(\frac{1}{\eps}\gamb^\eps (\ub(\eps))(t) - \gamb^0(t)\right)\Qtbb(\eps)^T \right)\\
& :  \left(\Qtbb(\eps)\left(\frac{1}{\eps}\gamb^\eps (\ub(\eps))(t)- \gamb^0(t)\right)\Qtbb(\eps)^T \right) \sqrt{g(\eps)} 
 +\frac{1}{2} \koefdva \int_\Omega (\pi(\eps)(t) - \pi^0(t))^2 \sqrt{g(\eps)} dz\\
&\quad + \eps^2 \int_0^t \int_\Omega (\nabla^\eps \pi(\eps) - \nabla^\eps \pi^0) \Qtbb(\eps)^T \cdot (\nabla^\eps \pi(\eps) - \nabla^\eps \pi^0) \Qtbb(\eps)^T \sqrt{g(\eps)} dz.
\endaligned
$$
We will show that $\Lambda(\eps) \to \Lambda$ as $\eps$ tends to zero. Since $\Lambda(\eps)\geq 0$ the $\Lambda\geq 0$ as well. After some calculation we will show that actually  $\Lambda=0$ and $\ogamb=0$. This will  imply the strong convergence in (\ref{konvergencije}).

Since we have only weak convergences in (\ref{konvergencije}) we first remove quadratic terms in $\Lambda(\eps)$ using (\ref{3dw3}) divided by $\eps^3$. Integration of (\ref{3dw3}) over time, using $\pi(\eps)|_{t=0}=0$ and $\ub(\eps)|_{t=0}=0$, implies
$$
\aligned
& \frac{1}{2}  \int_{\Omega} \cC \left(\Qtbb(\eps)\frac{1}{\eps}\gamb^\eps (\ub(\eps))\Qtbb(\eps)^T \right) {  :}  \left(\Qtbb(\eps) \frac{1}{\eps}\gamb^\eps (\ub(\eps)) \Qtbb(\eps)^T\right) \sqrt{g(\eps)} dz\\
& \qquad
+ \frac{1}{2} \koefdva \int_{\Omega}  \pi(\eps)^2 \sqrt{g(\eps)} dz
+ \eps^2 \int_0^t \int_{\Omega} {  \Qtbb(\eps)} \nabla^\eps  \pi(\eps)  \cdot {  \Qtbb(\eps)} \nabla^\eps  \pi(\eps)   \sqrt{g(\eps)} dz d\tau\\
& = \int_0^t \int_{\Sigma_{\pm}}  {\mathcal P}_\pm  \cdot \parder{\ub(\eps)}{t} \sqrt{g(\eps)}ds  d\tau   \mp \int_0^t \int_{\Sigma_{\pm}} \Vred  \pi(\eps) \sqrt{g(\eps)} ds d\tau.
\endaligned
$$
Inserting this into the definition of $\Lambda(\eps)$ we obtain
$$
\aligned
\Lambda (\eps)(t) &=
\int_0^t \int_{\Sigma_{\pm}}  {\mathcal P}_\pm  \cdot \parder{\ub(\eps)}{t} \sqrt{g(\eps)}ds d\tau \mp \int_0^t \int_{\Sigma_{\pm}} \Vred  \pi(\eps) \sqrt{g(\eps)} ds d\tau \\
&\quad - \int_{\Omega} \cC \left(\Qtbb(\eps)\frac{1}{\eps}\gamb^\eps (\ub(\eps))(t)\Qtbb(\eps)^T \right) {  :}  \left(\Qtbb(\eps) \gamb^0(t) \Qtbb(\eps)^T \right) \sqrt{g(\eps)} dz\\
&\quad -  \koefdva \int_{\Omega}  \pi(\eps) \pi^0 \sqrt{g(\eps)} dz\\
&\quad - 2 \eps^2 \int_0^t \int_{\Omega} {  \Qtbb(\eps)} \nabla^\eps  \pi(\eps)(t)  \cdot {  \Qtbb(\eps)}\nabla^\eps  \pi^0(t)   \sqrt{g(\eps)} dz d\tau\\
&\quad + \frac{1}{2}\int_{\Omega} \cC \left(\Qtbb(\eps) \gamb^0 \Qtbb(\eps)^T \right) {  :}  \left(\Qtbb(\eps)\gamb^0\Qtbb(\eps)^T \right) \sqrt{g(\eps)} dz
+\frac{1}{2}  \koefdva \int_{\Omega}  (\pi^0)^2 \sqrt{g(\eps)} dz\\
&\quad + \eps^2 \int_0^t \int_{\Omega} {  \Qtbb(\eps)} \nabla^\eps  \pi^0  \cdot {  \Qtbb(\eps)} \nabla^\eps  \pi^0   \sqrt{g(\eps)} dz d\tau.
\endaligned
$$
Now we take the limit as $\eps$ tends to zero and obtain that $\Lambda(\eps) \to \Lambda \geq 0$, where
\begin{equation}\label{Lambda}
\aligned
&\Lambda =
\int_0^t \int_{\omega}  ({\mathcal P}_+ + \cP_-)  \cdot \parder{\ub}{t} \sqrt{a}ds d\tau \mp \int_0^t \int_{\Sigma_{\pm}} \Vred  \pi^0 \sqrt{a} ds d\tau \\
& - \frac{1}{2} \int_{\Omega} \cC \left(\Qtbb \gamb^0 \Qtbb^T \right)  :  \left(\Qtbb \gamb^0 \Qtbb^T \right) \sqrt{a} dz
-  \frac{1}{2}\koefdva \int_{\Omega}  (\pi^0)^2 \sqrt{a} dz - \int_0^t \int_{\Omega} \left(\parder{\pi^0}{z_3}\right)^2 \sqrt{a} dz d\tau.
\endaligned
\end{equation}
We now insert $\displaystyle \parder{\ub}{t}$ as a test function in (\ref{druga}), $\pi^0$ in (\ref{treca}) and sum up the equations.
$$
\aligned
&  \frac{\hnatri}{12} \frac{1}{2} \frac{d}{dt} \int_\omega \tcC (\Abb^c\rhob(\ub)) :  (\rhob(\ub)\Abb^c) \sqrt{a} dz_1 dz_2 + \frac{2 \mut  \alpha}{\lambdat+2\mut}\int_\omega   \bigg( \int_{-\hnaj/2}^{\hnaj/2} z_3 \pi^0 dz_3  \bigg) \Abb^c  : \rhob( \ub) \sqrt{a}dz_1 dz_2\\
&  + \frac{1}{2} \frac{d}{dt} \int_{\Omega} \left(\koefdva + \frac{\alpha^2}{\lambdat+2\mut}\right) (\pi^0)^2 \sqrt{a} dz +   \int_{\Omega} \alpha \parder{}{t} \left( \frac{2\mut}{\lambdat+2\mut} \Abb^c  :  \ogamb - z_3 \frac{2\mut}{\lambdat+2\mut} \Abb^c  :  \rhob(\ub)\right)  \pi^0 \sqrt{a} dz\\
&\qquad +  \int_{\Omega} \left(\parder{\pi^0}{z_3}\right)^2   \sqrt{a} dz =   \int_{\omega}  ({\mathcal P}_+ + {\mathcal P}_-)  \cdot \parder{\ub}{t} \sqrt{a} dz_1 dz_2
\mp \int_{\Sigma_{\pm}} \Vred  \pi^0 \sqrt{a} ds.
\endaligned
$$
The anti-symmetric terms cancel out as before. We integrate the equation over time and use the initial conditions to obtain
\begin{equation}\label{1}
\aligned
&  \frac{\hnatri}{12} \frac{1}{2} \int_\omega \tcC (\Abb^c\rhob(\ub)) {  :}  (\rhob(\ub)\Abb^c) \sqrt{a} dz_1 dz_2 +  \frac{1}{2} \int_{\Omega} \left(\koefdva + \frac{\alpha^2}{\lambdat+2\mut}\right) (\pi^0)^2 \sqrt{a} dz\\
&\qquad +  \int_0^t \int_{\Omega} \left(\parder{\pi^0}{z_3}\right)^2   \sqrt{a} dz d\tau +   \int_0^t \int_{\Omega} \alpha  \frac{2\mut}{\lambdat+2\mut} \Abb^c {  :}  \parder{\ogamb}{t} \pi^0 \sqrt{a} dz d\tau\\
&\quad  =  \int_0^t \int_{\omega}  ({\mathcal P}_+ + {\mathcal P}_-)  \cdot \parder{\ub}{t} \sqrt{a} ds d\tau
\mp \int_0^t \int_{\Sigma_{\pm}} \Vred  \pi^0 \sqrt{a} ds d\tau .
\endaligned
\end{equation}
Next we compute the elastic energy
\allowdisplaybreaks[4]
\begin{align*}
&\hspace{-3ex}\int_{\Omega} \cC \left(\Qtbb \gamb^0 \Qtbb^T \right)  :  \left(\Qtbb \gamb^0 \Qtbb^T \right) \sqrt{a} dz\\
&=  \int_{\Omega} \lambdat (\tr{\left(\Qtbb \gamb^0\Qtbb^T \right)})^2 + 2 \mut \Qtbb^T \Qtbb\gamb^0\Qtbb^T \Qtbb  :   \gamb^0 \sqrt{a} dz\\
&= \int_{\Omega} \lambdat \left(\frac{2\mut}{\lambdat+2\mut} \Abb^c  :  \ogamb - z_3 \frac{2\mut}{\lambdat+2\mut} \Abb^c  : \rhob(\ub) +
\frac{\alpha}{\lambdat+2\mut} \pi^0\right)^2\sqrt{a} dz\\
&\quad + \int_\Omega 2 \mut \bigg(
\Abb^c (\ogamb - z_3 \rhob(\ub)) \Abb^c  :  (\ogamb - z_3 \rhob(\ub))
+ \left( \frac{\alpha}{\lambdat+2\mut} \pi^0 -\frac{\lambdat}{\lambdat+2\mut} \Abb^c  :  (\ogamb - z_3 \rhob(\ub))\right)^2 \bigg) \sqrt{a} dz\\
&=   \int_{\Omega} \lambdat \bigg(2 \frac{\alpha}{\lambdat+2\mut} \frac{2\mut}{\lambdat+2\mut} \pi^0 \Abb^c  :  \ogamb - 2 \frac{\alpha}{\lambdat+2\mut} \frac{2\mut}{\lambdat+2\mut} \pi^0 z_3  \Abb^c  :  \rhob(\ub)\\
&\qquad + \left(\frac{2\mut}{\lambdat+2\mut}\right)^2 (\Abb^c  :  \ogamb)^2 + (z_3)^2 \left(\frac{2\mut}{\lambdat+2\mut}\right)^2 \left( \Abb^c  :  \rhob(\ub)\right)^2 +
\left(\frac{\alpha}{\lambdat+2\mut}\right)^2 (\pi^0)^2\bigg)\sqrt{a} dz\\
&\quad + \int_\Omega 2 \mut \bigg(
\Abb^c \ogamb  \Abb^c  :  \ogamb  + (z_3)^2 \Abb^c \rhob(\ub) \Abb^c  :   \rhob(\ub)\\
&\qquad  - 2 \frac{\alpha}{\lambdat+2\mut} \frac{\lambdat}{\lambdat+2\mut} \pi^0  \Abb^c  :  \ogamb + 2 \frac{\alpha}{\lambdat+2\mut}  \frac{\lambdat}{\lambdat+2\mut} \pi^0 z_3 \Abb^c  :   \rhob(\ub)\\
&\qquad + \left(\frac{\alpha}{\lambdat+2\mut}\right)^2 (\pi^0)^2 + \left(\frac{\lambdat}{\lambdat+2\mut} \right)^2 (\Abb^c  :  \ogamb)^2 + (z_3)^2 \left(\frac{\lambdat}{\lambdat+2\mut}\right)^2 (\Abb^c  :   \rhob(\ub))^2 \bigg) \sqrt{a} dz\\
&=  \int_\Omega \bigg( \left(\frac{\lambdat (2 \mut)^2}{(\lambdat+2\mut)^2} + \frac{2\mut\lambdat^2}{(\lambdat+2\mut)^2}  \right)  (\Abb^c  :  \ogamb)^2 + 2\mut \Abb^c \ogamb  \Abb^c  :  \ogamb\\
&\qquad + (z_3)^2\left( \left(\frac{\lambdat(2\mut)^2}{(\lambdat+2\mut)^2} + \frac{2\mut \lambdat^2}{(\lambdat+2\mut)^2}\right)\left( \Abb^c  :  \rhob(\ub)\right)^2 + 2\mut\Abb^c \rhob(\ub) \Abb^c  :   \rhob(\ub)\right)\\
&\qquad +\left(  \lambdat \frac{\alpha^2}{(\lambdat+2\mut)^2} + 2\mut \frac{\alpha^2}{(\lambdat+2\mut)^2}  \right) (\pi^0)^2
\bigg)\sqrt{a} dz\\
&\quad +\int_\Omega \bigg( \left(  2\lambdat \alpha \frac{2\mut}{(\lambdat+2\mut)^2}  - 4\mut \alpha \frac{\lambdat}{(\lambdat+2\mut)^2}\right)  \pi^0 \Abb^c  :  \ogamb\\
&\qquad + \left(  - \lambdat  2 \frac{\alpha}{\lambdat+2\mut} \frac{2\mut}{\lambdat+2\mut} + 4\mut  \frac{\alpha}{\lambdat+2\mut}  \frac{\lambdat}{\lambdat+2\mut}\right)\pi^0 z_3 \Abb^c  :   \rhob(\ub)\bigg)\sqrt{a} dz\\
&=  \int_\Omega \left( \tcC (\Abb^c \ogamb)   :  \ogamb \Abb^c + (z_3)^2 \tcC (\Abb^c \rhob(\ub))  :   \rhob(\ub)\Abb^c  + \frac{\alpha^2}{\lambdat+2\mut} (\pi^0)^2\right)\sqrt{a} dz.
\end{align*}
Insertion of the above equality and (\ref{1}) in (\ref{Lambda}), yields
$$
\aligned
\Lambda &=
\frac{\hnatri}{12} \frac{1}{2} \int_\omega \tcC (\Abb^c\rhob(\ub)) {  :}  (\rhob(\ub)\Abb^c) \sqrt{a} dz_1 dz_2 +  \frac{1}{2} \int_{\Omega} \left(\koefdva + \frac{\alpha^2}{\lambdat+2\mut}\right) (\pi^0)^2 \sqrt{a} dz\\
&\quad +  \int_0^t \int_{\Omega} \left(\parder{\pi^0}{z_3}\right)^2   \sqrt{a} dz d\tau +   \int_0^t \int_{\Omega} \alpha  \frac{2\mut}{\lambdat+2\mut} \Abb^c {  :}  \parder{\ogamb}{t} \pi^0 \sqrt{a} dz d\tau\\
&\quad - \frac{1}{2} \int_\Omega \left( \tcC (\Abb^c \ogamb)   {  :}  \ogamb \Abb^c + (z_3)^2 \tcC (\Abb^c \rhob(\ub)) {  :}   \rhob(\ub)\Abb^c  + \frac{\alpha^2}{\lambdat+2\mut} (\pi^0)^2\right)\sqrt{a} dz \\
&\quad -  \frac{1}{2}\koefdva \int_{\Omega}  (\pi^0)^2 \sqrt{a} dz - \int_0^t \int_{\Omega} \left(\parder{\pi^0}{z_3}\right)^2 \sqrt{a} dz d\tau\endaligned $$
$$
\aligned
&=
\int_0^t \int_{\Omega} \alpha  \frac{2\mut}{\lambdat+2\mut} \Abb^c {  :} \parder{\ogamb}{t} \pi^0 \sqrt{a} dz d\tau - \frac{1}{2} \int_\Omega \left( \tcC (\Abb^c \ogamb)   {  :}  \ogamb \Abb^c\right)\sqrt{a} dz.
\endaligned
$$
From (\ref{treca}) for $q$ independent of transversal variable we obtain 
$$
\aligned
&   \parder{}{t}\int_{\Omega} \left( \left(\koefdva + \frac{\alpha^2}{\lambdat+2\mut}\right)\pi^0 + \alpha  \frac{2\mut}{\lambdat+2\mut} \Abb^c  :  \ogamb\right)  q \sqrt{a} dz  =  0, \qquad q \in H^1(\omega).
\endaligned
$$
This implies that
\begin{equation}\label{pi0}
\alpha  \frac{2\mut}{\lambdat+2\mut} \Abb^c {  :}  \parder{\ogamb}{t} = - \left(\koefdva + \frac{\alpha^2}{\lambdat+2\mut}\right)\parder{}{t} \int_{-\hnaj/2}^{\hnaj/2} \pi^0 dz_3.
\end{equation}
Inserting (\ref{pi0}) into $\Lambda$ yields
\begin{equation}\label{lambda1}
\aligned
\Lambda &= - \left(\koefdva + \frac{\alpha^2}{\lambdat+2\mut}\right)  \int_0^t \int_{\Omega} \parder{}{t} \left(\int_{-\hnaj/2}^{\hnaj/2} \pi^0 dz_3\right) \pi^0 \sqrt{a} dz d\tau - \frac{1}{2} \int_\Omega \left( \tcC (\Abb^c \ogamb)    :  \ogamb \Abb^c\right)\sqrt{a} dz\\
&= - \left(\koefdva + \frac{\alpha^2}{\lambdat+2\mut}\right)  \frac{1}{2} \int_0^t \frac{d}{dt}\int_{\omega} \left(\int_{-\hnaj/2}^{\hnaj/2} \pi^0 dz_3\right)^2 \sqrt{a} dz_1 dz_2 d\tau - \frac{1}{2} \int_\Omega \left( \tcC (\Abb^c \ogamb)    :  \ogamb \Abb^c\right)\sqrt{a} dz\\
&= - \left(\koefdva + \frac{\alpha^2}{\lambdat+2\mut}\right)  \frac{1}{2} \int_{\omega} \left(\int_{-\hnaj/2}^{\hnaj/2} \pi^0 dz_3\right)^2 \sqrt{a} dz_1 dz_2 d\tau - \frac{1}{2} \int_\Omega \left( \tcC (\Abb^c \ogamb)    :  \ogamb \Abb^c\right)\sqrt{a} dz,
\endaligned
\end{equation}
where in the last equation we have used 
  that $\pi^0|_{t=0}=0$.
Since $\Lambda \geq 0$ by definition and since the right hand side is nonpositive we conclude that $\Lambda=0$. Positivity of $\tcC$ implies $\ogamb=0$ and thus the strain $\gamb^0$ is fully determined by $\ub$ and $\pi^0$
$$
\gamb^0 = \left[\begin{array}{cc}
-z_3 \rhob(\ub) &\begin{array}{c}
0\\0\end{array}\\
\begin{array}{cc}
0 & 0
\end{array}
&
\frac{\alpha}{\lambdat+2\mut} \pi^0 + z_3 \frac{\lambdat}{\lambdat+2\mut} \Abb^c  :  \rhob(\ub)
\end{array}\right].
$$
Moreover, from (\ref{pi0}) we obtain that $\int_{-\hnaj/2}^{\hnaj/2} \pi^0 =0$ and thus the poroelastic flexural shell model is given by
\begin{equation}\label{druga10}
\aligned
&  \frac{\hnatri}{12} \int_\omega \tcC (\Abb^c\rhob(\ub))  :  \rhob(\vb)\Abb^c \sqrt{a} dz_1 dz_2 + \frac{2 \mut  \alpha}{\lambdat+2\mut}\int_\omega  \int_{-\hnaj/2}^{\hnaj/2} z_3 \pi^0 dz_3  \Abb^c  :  \rhob(\vb) \sqrt{a}dz_1 dz_2\\
 &\quad =   \int_{\omega}  ({\mathcal P}_+ + {\mathcal P}_-)  \cdot \vb \sqrt{a} dz_1 dz_2, \qquad \vb \in \cV_F(\omega).
\endaligned
\end{equation}
\begin{equation}\label{treca10}
\aligned
&   \int_{\Omega} \left(\koefdva + \frac{\alpha^2}{\lambdat+2\mut}\right)\parder{\pi^0}{t} q \sqrt{a} dz -  \int_{\Omega} \alpha    \frac{2\mut}{\lambdat+2\mut} \Abb^c  :  \rhob(\parder{\ub}{t}) z_3 q \sqrt{a} dz +  \int_{\Omega} \parder{\pi^0}{z_3} \parder{q}{z_3}   \sqrt{a} dz\\
& \qquad =  \mp \int_{\Sigma_{\pm}} \Vred  q \sqrt{a} ds, \qquad q \in H^1(\Omega).
\endaligned
\end{equation}
We now have $\Lambda(\eps)(t) \to 0$ for every $t\in [0,T]$. Since $\Lambda(\eps): [0,T] \to \ZR$ is an equicontinuous family, 
we conclude strong convergences of the strain tensor and the pressure
\begin{equation}\label{jake1}
\aligned
&\frac{1}{\eps}\gamb^\eps (\ub(\eps)) \to \gamb^0 \qquad \mbox{ strongly in } C([0,T];L^2(\Omega;\ZR^{3\times 3})),\\
& \pi (\eps) \to \pi^0 \qquad \mbox{ strongly in } C([0,T];L^2(\Omega)),\\
&\parder{\pi(\eps)}{z_3} \to \parder{\pi^0}{z_3} \qquad \mbox{ strongly in } L^2(0,T;L^2(\Omega)).
\endaligned
\end{equation}

\textsc{Step 4} (Strong convergence for displacements).
We prove the strong convergence in two steps. In the first step we use the last part of Theorem~\ref{tglavni} and prove pointwise convergence of $\ub(\eps)$. Due to equicontinuity it then implies the uniform convergence, i.e., we obtain
$$
\ub(\eps) \to \ub\qquad \mbox{ strongly in } C([0,T]; H^1(\Omega;\ZR^3)).
$$

\begin{lemma}\label{l1}
$$
\rhob(\ub(\eps)) \to \rhob(\ub) \qquad \mbox{ strongly in } C([0,T];H^{-1}(\Omega;\ZR^{2\times 2})).
$$
\end{lemma}
\begin{proof}
From \cite[Theorem 5.2-1]{Ciarlet3} for $\alpha, \beta =1,2$ we have the estimate (for a.e. $t \in [0,T]$)
\begin{equation}\label{estC}
\aligned
&\|\frac{1}{\eps} \parder{}{z_3} \gamma^\eps_{\alpha\beta}(\ub(\eps)) + \rho_{\alpha\beta}(\ub(\eps))\|_{H^{-1}(\Omega)}\\
&\qquad \leq C (\sum_{i=1}^3 \|\gamma^\eps_{i3}(\ub(\eps))\|_{L^2(\Omega)} + \eps \|u_1(\eps)\|_{L^2(\Omega)} + \eps \|u_2 (\eps)\|_{L^2(\Omega)}+ \eps \|u_3(\eps)\|_{H^1(\Omega)}),
\endaligned
\end{equation}
for $C$ independent of $\eps$.
By the strong convergence (\ref{jake1}) scaled transformed symmetrized gradient
$\frac{1}{\eps}\gamb^\eps (\ub(\eps))$ is bounded in $C([0,T]; L^2(\Omega;\ZR^{3\times 3}))$
and by the a priori estimates from Corollary~\ref{c1} $\ub(\eps)$ is bounded in $C([0,T]; H^1(\Omega;\ZR^3))$. Therefore the right hand side in (\ref{estC}) tends to zero uniformly with respect to $t\in[0,T]$. Thus we obtain
\begin{equation}\label{jakaskoro}
\frac{1}{\eps} \parder{}{z_3} \gamma^\eps_{\alpha\beta}(\ub(\eps)) + \rho_{\alpha\beta}(\ub(\eps)) \to 0 \qquad \mbox{ strongly in } C([0,T]; H^{-1} (\Omega)), \quad \alpha, \beta=1,2.
\end{equation}
From (\ref{jake1}) we have that
$$
\frac{1}{\eps} \parder{}{z_3}\gamb^\eps(\ub(\eps)) \to \parder{}{z_3}\gamb^0  \qquad \mbox{ strongly in } C([0,T]; H^{-1} (\Omega;\ZR^{3\times 3})).
$$
Using this convergence in (\ref{jakaskoro}) we obtain that functions $\rhob(\ub(\eps))$ converge strongly in $C([0,T]; H^{-1}(\Omega;\ZR^{2\times 2}))$.
\qed
\end{proof}

\begin{theorem}\label{tjaka}
$$
\ub(\eps) \to \ub \qquad \mbox{ strongly in } C([0,T]; H^1(\Omega;\ZR^3)).
$$
\end{theorem}
\begin{proof}
From Lemma~\ref{l1} and Remark~\ref{r*}  we have pointwise convergences
$$
\aligned
&\rhob(\ub(\eps))(t) \to \rhob(\ub)(t) \qquad \mbox{ strongly in } H^{-1}(\Omega;\ZR^{2\times 2}),\\
&\ub(\eps)(t) \rightharpoonup \ub(t)\qquad \mbox{ weakly in } H^{1}(\Omega;\ZR^3),
\endaligned
$$
for every $t\in [0,T]$. Thus the last part of the Theorem~\ref{tglavni} (taken from \cite{Ciarlet3}) implies
$$
\ub(\eps)(t) \rightharpoonup \ub(t)\qquad \mbox{ strongly in } H^{1}(\Omega;\ZR^3),
$$
for all $t \in [0,T]$. This implies that the function $w(\eps):[0,T]\to \ZR$ given by
$$
w(\eps)(t) = \|\ub(\eps)(t) - \ub(t)\|_{H^1(\Omega;\ZR^3)}
$$
converges pointwisely to zero. From the uniform estimate of $\ub(\eps)$ in $H^1(0,T;H^1(\Omega;\ZR^3))$ we obtain equicontinuity of the family $(w(\eps))_{\eps>0}$. This implies the uniform convergence of $w(\eps)$. This implies the statement of the theorem.
\qed
\end{proof}

\textsc{Step 5} (Strong convergence for the stress tensor).
As a consequence of the convergence of the term $\frac{1}{\eps}\gamb^\eps(\ub(\eps))$ from (\ref{jake1}) we obtain the convergence of the scaled stress tensor.

\smallskip

\noindent{\bf Proof of Corollary \ref{cstress11}.}
We compute $\sigma$ in the local basis given by $\Qbb$ using (\ref{trtgama}) and (\ref{gama}) and that $\ogamb =0$. We obtain
$$
\aligned
\Qtbb^T \sigma \Qtbb &= \lambdat \tr{(\Qtbb \gamb^0 \Qtbb^T)} \Qtbb^T \Qtbb + 2 \mut \Qtbb^T \Qtbb \gamb^0 \Qtbb^T \Qtbb - \alpha \pi^0 \Qtbb^T\Qtbb\\
&=  \left(\frac{\alpha\lambdat}{\lambdat+2\mut} \pi^0 -  z_3 \frac{2\mut \lambdat}{\lambdat+2\mut} \Abb^c {  :}  \rhob(\ub) - \alpha \pi^0\right)\left[\begin{array}{cc}
\Abb^c &0\\
0 & 1 \end{array}\right]\\
&\quad + 2 \mut  \left[\begin{array}{cc}
- z_3 \Abb^c \rhob(\ub) \Abb^c & 0 \\
0 & \frac{\alpha}{\lambdat+2\mut} \pi^0 + z_3 \frac{\lambdat}{\lambdat+2\mut} \Abb^c {  :}  \rhob(\ub))
\end{array}\right]\\
&=  \left(-\frac{2 \alpha}{\lambdat+2\mut} \pi^0 -  z_3 \frac{2\mut \lambdat}{\lambdat+2\mut} \Abb^c {  :}  \rhob(\ub) \right)\left[\begin{array}{cc}
\Abb^c &0\\
0 & 0 \end{array}\right] -  \mut  \left[\begin{array}{cc}
2 z_3 \Abb^c \rhob(\ub) \Abb^c & 0 \\
0 & 0
\end{array}\right]\\
&=  \left[\begin{array}{cc}
-\frac{2 \alpha}{\lambdat+2\mut} \pi^0 \Abb^c -  z_3 \left( \frac{2\mut \lambdat}{\lambdat+2\mut} (\Abb^c {  :}  \rhob(\ub)) \Ibb + 2 \Abb^c \rhob(\ub)\right) \Abb^c &0\\
0 & 0 \end{array}\right].
\endaligned
$$
\qed




\section{Appendix}
\subsection{Properties of the metric tensor, the curvature tensor and the third fundamental form}\label{append}

Some symmetry properties of geometric coefficients are listed in the following lemma. For the proof see \cite{CiarletDG}.
\begin{lemma}\label{lbaza}
The following symmetries hold ($\alpha,\beta,\kappa \in\{1,2\}$)
$$
a_{\alpha \beta}=a_{\beta \alpha},
\quad
a^{\alpha \beta}=a^{\beta \alpha},
\quad
b_{\alpha \beta}=b_{\beta \alpha},
\quad
\Gamma^\kappa_{\alpha \beta}=\Gamma^\kappa_{\beta \alpha}.
$$
The change from the basis to basis is done using
$$
\ab_\alpha=\sum_{\kappa=1}^2 a_{\alpha \kappa}  \ab^\kappa, \qquad
\ab^\alpha=\sum_{\kappa=1}^2 a^{\alpha \kappa}  \ab_\kappa.
$$
Moreover, one has
$$
b^\tau_\alpha|_\beta=b^\tau_\beta|_\alpha, \qquad \sum_{\kappa=1}^2 b^\kappa_\alpha b_{\kappa \beta} = \sum_{\kappa=1}^2 b^\kappa_\beta b_{\kappa \alpha}.
$$
\end{lemma}

\subsection{Computation of $\Thetab(\vb)$}
\begin{lemma}\label{lTheta}
For $\vb \in \cV_F(\omega)$ and $\Thetab(\vb)$ defined in the Step 2 of the convergence proof (Section~\ref{S6}) one has
$$
\Thetab(\vb) = - z_3  \left[\begin{array}{cc}
\rhob(\vb)   & \begin{array}{c} \displaystyle \sum_{\kappa=1}^2 b^\kappa_1 (\partial_\kappa v_3 +  v_\tau b^\tau_\kappa) \\ \displaystyle \sum_{\kappa=1}^2 b^\kappa_2 (\partial_\kappa v_3 +  v_\tau b^\tau_\kappa)
\end{array}\\
\begin{array}{cc}
\displaystyle \sum_{\kappa=1}^2 b^\kappa_1 ( \partial_\kappa v_3 +  v_\tau b^\tau_\kappa ) & \displaystyle \sum_{\kappa=1}^2 b^\kappa_2 (\partial_\kappa v_3 + v_\tau b^\tau_\kappa)
\end{array} & 0
\end{array}\right].
$$
\end{lemma}
\begin{proof}
Let us first denote
$$
\Xib = \sum_{\kappa=1}^2 \left[\begin{array}{cc}
\partial_1(\frac{1}{2}\partial_1 v_3   + 2v_\kappa b^\kappa_1) & \frac{1}{2}\left( \partial_2 (\frac{1}{2}\partial_1 v_3   + 2v_\kappa b^\kappa_1) + \partial_1(\frac{1}{2}\partial_2 v_3   +2 v_\kappa b^\kappa_2)\right) \\
\frac{1}{2}\left( \partial_2 (\frac{1}{2}\partial_1 v_3   +2 v_\kappa b^\kappa_1) + \partial_1(\frac{1}{2} \partial_2 v_3   + 2v_\kappa b^\kappa_2)\right) & \partial_2(\frac{1}{2}\partial_2 v_3   +2 v_\kappa b^\kappa_2)
\end{array}\right]
$$
Then using (\ref{CS}) and (\ref{CS2}) we obtain
$$
\aligned
\Thetab(\vb) =&\  z_3 \Bigg( -\Xib + \sum_{\kappa=1}^2 (\partial_\kappa v_3   + 2 \sum_{\tau=1}^2 v_\tau b^\tau_\kappa)\left[\begin{array}{ccc}
\Gamma^\kappa_{11}   & \Gamma^\kappa_{12}    & -b^\kappa_1  \\
\Gamma^\kappa_{21}   & \Gamma^\kappa_{22}  & -b^\kappa_2 \\
-b^\kappa_1   & -b^\kappa_2 &0
\end{array} \right]\\
&\ + \sum_{\kappa=1}^2 v_\kappa \left[\begin{array}{ccc}
 b^\kappa_1|_1  &  b^\kappa_2|_1 & b^\tau_1 b^\kappa_\tau  \\
 b^\kappa_1|_2  &  b^\kappa_2|_2 & b^\tau_2 b^\kappa_\tau  \\
 b^\tau_1 b^\kappa_\tau   & b^\tau_2 b^\kappa_\tau &0
\end{array} \right] + v_3 \sum_{\kappa=1}^2 \left[\begin{array}{ccc}
 b^\kappa_1 b_{\kappa 1}  & b^\kappa_1 b_{\kappa 2}   & 0  \\
 b^\kappa_2 b_{\kappa 1}  & b^\kappa_2 b_{\kappa 2} & 0 \\
0   & 0 &0
\end{array} \right]
\Bigg)\\
=&\  - z_3 \Bigg( \sum_{\kappa=1}^2 \left[\begin{array}{ccc}
\frac{1}{2} \partial_1^2 v_3   + 2\partial_1 v_\kappa b^\kappa_1 & \frac{1}{2}  \partial_1 \partial_2 v_3   + \partial_2 v_\kappa b^\kappa_1  +  \partial_1 v_\kappa b^\kappa_2   & 0\\
\frac{1}{2}  \partial_1 \partial_2 v_3   + \partial_2 v_\kappa b^\kappa_1  +  \partial_1 v_\kappa b^\kappa_2 & \frac{1}{2} \partial_2^2 v_3   +2 \partial_2 v_\kappa b^\kappa_2  & 0\\
0 & 0 & 0
\end{array}\right]\\
&\ + \sum_{\kappa=1}^2 v_\kappa \left[\begin{array}{ccc}
2 \partial_1 b^\kappa_1 &    \partial_2 b^\kappa_1 +  \partial_1 b^\kappa_2  & 0\\
 \partial_2 b^\kappa_1  +   \partial_1 b^\kappa_2  &  2  \partial_2 b^\kappa_2 & 0\\
0 & 0 & 0
\end{array}\right] - \sum_{\kappa=1}^2  \partial_\kappa v_3  \left[\begin{array}{ccc}
\Gamma^\kappa_{11}   & \Gamma^\kappa_{12}    & 0  \\
\Gamma^\kappa_{21}   & \Gamma^\kappa_{22}  & 0 \\
0   & 0 &0
\end{array} \right]\\
&\  - 2  \sum_{\kappa, \tau=1}^2 v_\kappa b^\kappa_\tau\left[\begin{array}{ccc}
\Gamma^\tau_{11}   & \Gamma^\tau_{12}    & 0  \\
\Gamma^\tau_{21}   & \Gamma^\tau_{22}  & 0 \\
0   & 0 &0
\end{array} \right]  - \sum_{\kappa=1}^2 v_\kappa \left[\begin{array}{ccc}
 b^\kappa_1|_1  &  b^\kappa_2|_1 & 0  \\
 b^\kappa_1|_2  &  b^\kappa_2|_2 & 0  \\
 0   & 0 &0
\end{array} \right]\\
&\ - v_3 \sum_{\kappa=1}^2 \left[\begin{array}{ccc}
 b^\kappa_1 b_{\kappa 1}  & b^\kappa_1 b_{\kappa 2}   & 0  \\
 b^\kappa_2 b_{\kappa 1}  & b^\kappa_2 b_{\kappa 2} & 0 \\
0   & 0 &0
\end{array} \right]
- \sum_{\kappa=1}^2 (\partial_\kappa v_3   + \sum_{\tau=1}^2 v_\tau b^\tau_\kappa)\left[\begin{array}{ccc}
0   & 0    & -b^\kappa_1  \\
0   & 0  & -b^\kappa_2 \\
-b^\kappa_1   & -b^\kappa_2 &0
\end{array} \right]
\Bigg).
\endaligned
$$
Expressing $\partial_\alpha b^\kappa_\beta$ using $b^\kappa_\beta|_\alpha$ and collecting all terms with $v_\kappa$ we obtain
$$
\aligned
\Thetab(\vb)
=
&\ - z_3 \bigg( \sum_{\kappa=1}^2 \left[\begin{array}{ccc}
\frac{1}{2}\partial_1^2 v_3   + 2\partial_1 v_\kappa b^\kappa_1 & \frac{1}{2}\partial_1 \partial_2 v_3   + \partial_2 v_\kappa b^\kappa_1  +  \partial_1 v_\kappa b^\kappa_2   & 0\\
\frac{1}{2}\partial_1 \partial_2 v_3   + \partial_2 v_\kappa b^\kappa_1  +  \partial_1 v_\kappa b^\kappa_2 & \frac{1}{2} \partial_2^2 v_3   +2 \partial_2 v_\kappa b^\kappa_2  & 0\\
0 & 0 & 0
\end{array}\right] \\
&\
+ \sum_{\kappa=1}^2 v_\kappa \left[\begin{array}{ccc}
b^\kappa_1|_1 - 2\sum_{\tau=1}^2\Gamma^\kappa_{1\tau} b^\tau_1 &    b^\kappa_1|_2 - \sum_{\tau=1}^2 \Gamma^\kappa_{2\tau} b^\tau_1 - \sum_{\tau=1}^2 \Gamma^\kappa_{1\tau} b^\tau_2 & 0\\
- \sum_{\tau=1}^2 \Gamma^\kappa_{2\tau} b^\tau_1 +  b^\kappa_2|_1 - \sum_{\tau=1}^2\Gamma^\kappa_{1\tau} b^\tau_2  &  b^\kappa_2|_2 - 2 \sum_{\tau=1}^2 \Gamma^\kappa_{2\tau} b^\tau_2 & 0\\
0 & 0 & 0
\end{array}\right] \endaligned $$
$$\aligned
&\ - \sum_{\kappa=1}^2 \partial_\kappa v_3   \left[\begin{array}{ccc}
\Gamma^\kappa_{11}   & \Gamma^\kappa_{12}    & 0  \\
\Gamma^\kappa_{21}   & \Gamma^\kappa_{22}  & 0 \\
0   & 0 &0
\end{array} \right]  - v_3 \sum_{\kappa=1}^2\left[\begin{array}{ccc}
 b^\kappa_1 b_{\kappa 1}  & b^\kappa_1 b_{\kappa 2}   & 0  \\
 b^\kappa_2 b_{\kappa 1}  & b^\kappa_2 b_{\kappa 2} & 0 \\
0   & 0 &0
\end{array} \right]\\
&\ + \sum_{\kappa=1}^2 (\partial_\kappa v_3   + \sum_{\tau=1}^2 v_\tau b^\tau_\kappa)\left[\begin{array}{ccc}
0   & 0    & b^\kappa_1  \\
0   & 0  & b^\kappa_2 \\
b^\kappa_1   & b^\kappa_2 &0
\end{array} \right]
\bigg).
\endaligned
$$
This implies the statement of the lemma.
\qed
\end{proof}
A part of the proof can be also find in \cite[Step 4 in Section 6.2]{Ciarlet3}.

\subsection{Cylindrical surface}

Let $\omega_{L}= (-L/2,L/2) \times (0,d)$
, where $d\in ( 0, 2\pi )$ and let $(z, \theta)$ denotes
the generic point in $\omega$. Let $R>0$. We define the
cylindrical shell by the parametrization
$$
\vphib : \omega_{L} \to \ZR^3, \qquad \vphib(z, \theta)=(R \cos{\theta}, R\sin{\theta},
z)^T.
$$
For $d=2\pi$ the surface is the full cylinder.
Then the extended covariant basis of the shell $S=\vphib(\omega)$
is given by
\begin{eqnarray*}
\ab_1(z, \theta) &=& \partial_z \vphib(z, \theta) = (0, 0, 1)^T,\\
\ab_2(z, \theta) &=& \partial_\theta \vphib(z, \theta) = (-R \sin{\theta}, R\cos{\theta}, 0)^T,\\
\ab_3(z, \theta) &=& \frac{\ab_1(z, \theta) \times \ab_2(z, \theta)}{|\ab_1(z, \theta) \times \ab_2(z, \theta)|} =
(-\cos{\theta}, -\sin{\theta}, 0)^T.
\end{eqnarray*}
The contravariant basis is biorthogonal and is defined by
\begin{eqnarray*}
\ab^1(z, \theta) &=&  (0, 0, 1)^T,\\
\ab^2(z, \theta) &=&  (-\frac{1}{R} \sin{\theta}, \frac{1}{R}\cos{\theta}, 0)^T,\\
\ab^3(z, \theta) &=& (-\cos{\theta}, -\sin{\theta}, 0)^T.
\end{eqnarray*}
The covariant $\Abb_c=(a_{\alpha \beta})$ and contravariant
$\Abb^c=(a^{\alpha \beta})$ metric tensors are respectively given
by
$$
\Abb_c =\left[ \begin{array}{cc} 1 & 0 \\ 0 &
R^2\end{array}\right], \qquad \Abb^c =\left[ \begin{array}{cc} 1 &
0 \\ 0 & \frac{1}{R^2}\end{array}\right]
$$
and the area element is now $\sqrt{a} dS = \sqrt{\det{\Abb_c}} dS=
R dS$. The covariant and mixed components of the curvature tensor
are now given by
\begin{equation*}
\aligned
&b_{1 1} = b_{1 2} = b_{2 1} = 0,\quad
b_{2 2} = R,\qquad b^{1}_{1}= b^{1}_{2} = b^{2}_{1} =0,\quad b^{2}_{2} = \frac{1}{R}.
\endaligned
\end{equation*}
A simple calculation shows
\begin{eqnarray*}
\Gamma^\sigma_{\alpha
\beta}&=&\ab^\sigma \cdot \partial_\beta \ab_\alpha=0,\qquad \alpha,\beta,\sigma\in \{1,2\},\\
b^\sigma_{\beta}|_{\alpha} &=& \partial_\alpha b^\sigma_\beta
+\Gamma^\sigma_{\alpha \tau} b^\tau_\beta - \Gamma^\tau_{\beta
\alpha} b^\sigma_\tau = \partial_\alpha b^\sigma_\beta = 0, \qquad \alpha,\beta,\sigma\in \{1,2\}.
\end{eqnarray*}

Now the displacement vector $\vtb$ in canonical coordinates is
rewritten in the local basis $\vtb=\Qbb\vb = v_1 \ab^1 + v_2
\ab^2 + v_3 \ab^3$. Note that contravariant basis is different than
the usual basis associated with the cylindrical coordinates. One has
$$
v_1 = v_z, \quad v_2 = R v_\theta, \quad v_3 = -v_r.
$$
Similarly, { ${\mathcal \Pt}_\pm=\Qbb^{-T}\cP_\pm = (\cP_\pm)_1 \ab_1 + (\cP_\pm)_2
\ab_2 + (\cP_\pm)_3 \ab_3$.} Thus
$$
{ (\cP_\pm)}_1 = { (\cP_\pm)}_z, \quad { (\cP_\pm)}_2 = \frac{1}{R}  { (\cP_\pm)}_\theta, \quad { (\cP_\pm)}_3 = -{ (\cP_\pm)}_r.
$$
Thus
$$
\cP_\pm \cdot \vb = { (\cP_\pm)}_1 v_1 + { (\cP_\pm)}_2 v_2 + { (\cP_\pm)}_3 v_3 = { (\cP_\pm)}_z v_z + { (\cP_\pm)}_\theta v_\theta + { (\cP_\pm)}_r v_r.
$$

Inserting {the geometry} coefficients into the strains $\gamb$ and
$\varrhob$ we obtain
$$
\aligned
\gamb(\vb) &= \left[\begin{array}{cc}
\partial_1 v_1  - \sum_{\kappa=1}^2 \Gamma^\kappa_{11} v_\kappa  -  b_{11} v_3
&
\frac{1}{2}(\partial_1 v_2 + \partial_2 v_1) - \sum_{\kappa=1}^2 \Gamma^\kappa_{12} v_\kappa  -  b_{12} v_3\\
\frac{1}{2}(\partial_1 v_2 + \partial_2 v_1) - \sum_{\kappa=1}^2 \Gamma^\kappa_{21} v_\kappa  -  b_{21} v_3
&
\partial_2 v_2  - \sum_{\kappa=1}^2 \Gamma^\kappa_{22} v_\kappa  -  b_{22} v_3
\end{array}\right]\\
&= \left[\begin{array}{cc}
\partial_z v_z
&
\frac{1}{2}(R \partial_z v_\theta + \partial_\theta v_z)  \\
\frac{1}{2}(R \partial_z v_\theta + \partial_\theta v_z)
&
R \partial_\theta v_\theta   +  R v_r
\end{array}\right],
\endaligned
$$
$$
\aligned
\rhob(\vb) &= {\left[\begin{array}{cc}
\partial_{11} v_3
&
\partial_{12} v_3 \\
 \partial_{21} v_3
&
\partial_{22} v_3
\end{array}\right]}\\
&\quad {+ \sum_{\kappa=1}^2\left[\begin{array}{cc}
 b^\kappa_1 \partial_1 v_\kappa  +  b^\kappa_1 \partial_1 v_\kappa  + b^\kappa_1|_1 v_\kappa -  b^\kappa_1 b_{\kappa1} v_3
&
 b^\kappa_2 \partial_1 v_\kappa  +   b^\kappa_1 \partial_2 v_\kappa  +   b^\kappa_1|_2 v_\kappa -  b^\kappa_1 b_{\kappa2} v_3 \\
 b^\kappa_1 \partial_2 v_\kappa  +   b^\kappa_2 \partial_1 v_\kappa  +   b^\kappa_2|_1 v_\kappa -  b^\kappa_2 b_{\kappa1} v_3
&
 b^\kappa_2 \partial_2 v_\kappa  +  b^\kappa_2 \partial_2 v_\kappa  +  b^\kappa_2|_2 v_\kappa -  b^\kappa_2 b_{\kappa2} v_3
\end{array}\right]}\\
&= \left[\begin{array}{cc}
- \partial_{zz} v_r
&
-\partial_{z\theta} v_r +  \partial_z v_\theta  \\
- \partial_{\theta z} v_r  +  \partial_z v_\theta
&
-\partial_{\theta\theta} v_r  + 2 \partial_\theta v_\theta     + v_r
\end{array}\right].
\endaligned
$$

As example we write the model on the space
\begin{equation}\label{V_F}
\aligned
\cV_F(\omega_{L}) &= \{(v_z,v_\theta,v_r)\in H^1(\omega_{L}) \times H^1(\omega_{L}) \times H^2(\omega_{L}):  (v_z,v_\theta,v_r)|_{\theta=0,d } = 0,\\
&\qquad \qquad \partial_{\theta} {v_r}|_{\theta=0,d} = 0, \quad
\partial_z v_z = \frac{1}{2}(R \partial_z v_\theta + \partial_\theta v_z) = R \partial_\theta v_\theta   +  R v_r = 0\}
\endaligned
\end{equation}
 which includes clamping boundary conditions only on two generatrices of the portion of the cylinder. For cylindrical shell fully clamped one has $\cV_F(\omega_L) = \{0\}$, and the shell behaves as the generalized membrane shell, see \cite[Section 5.8]{Ciarlet3}.
From the condition of inextensibility in (\ref{V_F}) we obtain
$$
v_z(\theta), \qquad \partial_z v_\theta = - \frac{1}{R} \partial_\theta v_z, \qquad v_r = - \partial_\theta v_\theta.
$$
Therefore (using notation $'=\partial_\theta$)
$$
v_z(\theta), \qquad v_\theta (z,\theta)= - \frac{z}{R} v_z'(\theta) + w_\theta(\theta), \qquad v_r (z,\theta)= \frac{z}{R} v_z''(\theta) - w_\theta'(\theta).
$$
Smoothness and boundary conditions for $v_z,v_\theta,v_r$ imply
$$
v_z \in H^4_0(0,d), \qquad w_\theta \in H^3_0 (0,d).
$$
Thus
$$
\aligned
\rhob(\vb)
&= \left[\begin{array}{cc}
0
&
-\frac{1}{R} (v_z''+ v_z)'   \\
-\frac{1}{R} (v_z''+ v_z)'
&
- \frac{z}{R} (v_z''+ v_z)''+ (w_\theta'' + w_\theta)'
\end{array}\right].
\endaligned
$$
Now we insert this into the model given by (\ref{main_druga})--(\ref{maincetvrta}) written in dimensional form
$$
\aligned
&  \frac{\ell^3}{12}  \int_{0}^d \int_{-L/2}^{L/2} \bigg(
\frac{2\mu  \lambda}{\lambda+2\mu} \frac{1}{R^4}\left(- \frac{z}{R} (u_z''+ u_z)''+ (\omega_\theta'' + \omega_\theta)'\right) \left(- \frac{z}{R} (v_z''+ v_z)''+ (w_\theta'' + w_\theta)'\right)\\
&\qquad + 2\mu \left(\frac{2}{R^4} (u_z''+ u_z)'(v_z''+ v_z)' +\frac{1}{R^4} (\frac{z}{R} (u_z''+ u_z)'' - (\omega_\theta'' + \omega_\theta)')(\frac{z}{R} (v_z''+ v_z)''- (w_\theta'' + w_\theta)')\right)
\bigg)  R dz d\theta\\
&\qquad + \frac{2 \mu  \alpha}{\lambda+2\mu} \int_{0}^d \int_{-L/2}^{L/2} \int_{-\ell/2}^{\ell/2} r p^0 dr  \frac{1}{R^2}\left(- \frac{z}{R} ( v_z''+ v_z)''+ (w_\theta'' + w_\theta)')\right) R dz d\theta\\
 &\quad =  \int_{0}^d \int_{-L/2}^{L/2}  \Big((({\mathcal P}^{+\ell}_L)_z +({\mathcal P}^{-\ell}_L)_z) v_z +(({\mathcal P}^{+\ell}_L)_\theta + ({\mathcal P}^{-\ell}_L)_\theta)(- \frac{z}{R} v_z' + w_\theta)\\
 &\qquad \qquad + (({\mathcal P}^{+\ell}_L)_r + ({\mathcal P}^{-\ell}_L)_r) (\frac{z}{R} v_z'' - w_\theta')\Big) R dz d\theta,
 \qquad  v_z \in H^4_0(0,d), w_\theta \in H^3_0(0,d).
\endaligned
$$
$$\aligned
&   \frac{d}{d t}\int_{-\ell/2}^{\ell/2}\int_{0}^d \int_{-L/2}^{L/2} \left(\beta_G + \frac{\alpha^2}{\lambda+2\mu}\right) p^0 q R dz d\theta dr\\
&\qquad - \frac{2 \alpha\mu}{\lambda+2\mu} \int_{-\ell/2}^{\ell/2}\int_{0}^d \int_{-L/2}^{L/2}     \frac{1}{R^2} \left(- \frac{z}{R} \partial_t (u_z''+ u_z)''+ \partial_t (\omega_\theta'' + \omega_\theta)'\right) r q R dz d\theta dr\\
&\qquad +  \frac{k}{\eta}\int_{-\ell/2}^{\ell/2}\int_{0}^d \int_{-L/2}^{L/2} \parder{p^0}{r} \parder{q}{r}   R dz d\theta dr =  \int_{0}^d \int_{-L/2}^{L/2} \Vred_L  (q(-\frac{\ell}{2}) - q(\frac{\ell}{2})) R dz d\theta\\
&\qquad \mbox{in} \quad \mathcal{D}' (0,T), \qquad q \in H^1( (-\ell/2 , \ell/2); L^2 (\omega_L)),\\
&  p^0 = 0 \qquad \mbox{at} \quad t=0.
\endaligned
$$
After integration over $z$ the first equation separates and we obtain the following problem: find
$\{ p^0 , u_z,\omega_\theta \} \in C ([0,T]; H^1 ( \Omega^\ell_L ) \times H^4_0(0,d)\times H^3_0(0,d) )$,  $\partial_{r} p^0 \in L^2 ((0,T)\times \Omega^\ell_L)$ satisfying the system
$$
\aligned
&  \frac{\ell^3}{12}  \int_{0}^d  \frac{1}{R^3} \left(
4 \mu\frac{\lambda + \mu}{\lambda+2\mu}  \frac{L^3}{12 R^2} (u_z''+ u_z)'' (v_z''+ v_z)''+ 4\mu L (u_z''+ u_z)'(v_z''+ v_z)'  \right)   d\theta\\
&\qquad -\frac{2 \mu  \alpha}{\lambda+2\mu} \frac{1}{R^2}\int_{0}^d \left(\int_{-L/2}^{L/2} \int_{-\ell/2}^{\ell/2} r p^0 dr z dz \right) ( v_z''+ v_z)'' d\theta\\
 &\quad =  \int_{0}^d \bigg( R \int_{-L/2}^{L/2}  (({\mathcal P}^{+\ell}_L)_z + ({\mathcal P}^{-\ell}_L)_z) dz v_z - \int_{-L/2}^{L/2} (({\mathcal P}^{+\ell}_L)_\theta + ({\mathcal P}^{-\ell}_L)_\theta) z dz v_z' \\
 &\qquad  + \int_{-L/2}^{L/2} (({\mathcal P}^{+\ell}_L)_r + ({\mathcal P}^{-\ell}_L)_r) z dz v_z'' \bigg) d\theta, \qquad v_z \in H^4_0(0,d),\\
&  \frac{\ell^3}{12}  \int_{0}^d  \frac{L}{R^3}
4\mu\frac{ \lambda + \mu}{\lambda+2\mu}   (\omega_\theta'' + \omega_\theta)' (w_\theta'' + w_\theta)'   d\theta + \frac{2 \mu  \alpha}{\lambda+2\mu} \frac{1}{R}\int_{0}^d \left( \int_{-L/2}^{L/2} \int_{-\ell/2}^{\ell/2} r p^0 dr dz \right) (w_\theta'' + w_\theta)'   d\theta\\
 &\quad =  \int_{0}^d \left(\int_{-L/2}^{L/2}  (({\mathcal P}^{+\ell}_L)_\theta + ({\mathcal P}^{-\ell}_L)_\theta) dz w_\theta + \int_{-L/2}^{L/2} (({\mathcal P}^{+\ell}_L)_r + ({\mathcal P}^{-\ell}_L)_r) dz w_\theta'\right) R d\theta, \qquad  w_\theta \in H^3_0(0,d),\\
&   \frac{d}{d t}\int_{-\ell/2}^{\ell/2}\int_{0}^d \int_{-L/2}^{L/2} \left(\beta_G + \frac{\alpha^2}{\lambda+2\mu}\right) p^0 q R dz d\theta dr\\
&\qquad - \frac{2 \alpha\mu}{\lambda+2\mu} \int_{-\ell/2}^{\ell/2}\int_{0}^d \int_{-L/2}^{L/2}     \frac{1}{R^2} \left(- \frac{z}{R} \partial_t (u_z''+ u_z)''+ \partial_t (\omega_\theta'' + \omega_\theta)'\right) r q R dz d\theta dr\\
&\qquad +  \frac{k}{\eta}\int_{-\ell/2}^{\ell/2}\int_{0}^d \int_{-L/2}^{L/2} \parder{p^0}{r} \parder{q}{r}   R dz d\theta dr =  \int_{0}^d \int_{-L/2}^{L/2} \Vred_L  (q(-\frac{\ell}{2}) - q(\frac{\ell}{2})) R dz d\theta\\
&\qquad \mbox{in} \quad \mathcal{D}' (0,T), \qquad q \in H^1( (-\ell/2 , \ell/2); L^2 (\omega_L)),\\
&  p^0 = 0 \qquad \mbox{at} \quad t=0.
\endaligned
$$
The terms in the shell equation appear in the classical model of linear model of cylindrical shells, see e.g. \cite{Hoefk} and \cite{Hoefbook}.

\end{document}